\documentclass[12pt]{amsart}

\usepackage{amsmath,amsthm,amssymb}

\usepackage{a4wide}
\usepackage{enumerate}

\newtheorem{theorem}{Theorem}[section]
\newtheorem{lemma}[theorem]{Lemma}
\newtheorem{proposition}[theorem]{Proposition}
\newtheorem{corollary}[theorem]{Corollary}
\newtheorem*{theoremA}{Theorem A}
\newtheorem*{theoremB}{Theorem B}

\theoremstyle{remark}
\newtheorem{remark}[theorem]{Remark}

\newtheorem{claim}{Claim}
\newtheorem*{claim*}{Claim}

\newcommand{\C}{\ensuremath{\mathbb{C}}}
\newcommand{\R}{\ensuremath{\mathbb{R}}}

\newcommand{\g}[1]{\ensuremath{\mathfrak{#1}}}
\DeclareMathOperator{\tr}{tr}

\DeclareMathOperator{\id}{id}
\DeclareMathOperator{\Ad}{Ad}
\DeclareMathOperator{\ad}{ad}
\DeclareMathOperator{\Exp}{Exp}

\begin{document}
\title{Polar actions on complex hyperbolic spaces}

\author[J.\ C.\ D\'{\i}az Ramos]{Jos\'{e} Carlos D\'{\i}az Ramos}
\author[M.\ Dom\'{\i}nguez V\'{a}zquez]{Miguel Dom\'{\i}nguez V\'{a}zquez}
\author[A.\ Kollross]{Andreas Kollross}

\address{Department of Mathematics, University of Santiago de Compostela, Spain}
\address{Instituto de Ciencias Matem\'{a}ticas (CSIC-UAM-UC3M-UCM), Madrid, Spain}
\address{Institut f\"{u}r Geometrie und Topologie, Universit\"{a}t Stuttgart, Germany}

\email{josecarlos.diaz@usc.es}
\email{miguel.dominguez@icmat.es}
\email{kollross@mathematik.uni-stuttgart.de}

\thanks{The first and second authors have been supported by projects EM2014/009, GRC2013-045, MTM2013-41335-P and MTM2016-75897-P with FEDER funds (Spain). The second author acknowledges financial support from the Spanish Ministry of Economy and Competitiveness, through the Severo Ochoa Programme for Centres of Excellence in R\&D (SEV-2015-0554) and a Juan de la Cierva-formaci\'{o}n fellowship.}

\begin{abstract}
We classify polar actions on complex hyperbolic spaces up to
orbit \mbox{equivalence}.
\end{abstract}




\maketitle


\section{Introduction and main results}
\label{sect:Intro}

A proper isometric Lie group action on a Riemannian manifold is
called \emph{polar} if there exists an immersed connected submanifold that
meets every orbit orthogonally. Such a submanifold is then called a
\emph{section} of the action. In the special case where the section is flat in its induced Riemannian metric, the action is called \emph{hyperpolar}.
In this article, we classify polar actions on complex hyperbolic spaces.

The motivation for our work can be traced back to the work of Dadok~\cite{Da85}, who classified polar representations on Euclidean spaces, and the paper of Palais and Terng~\cite{pt}, who proved fundamental properties of polar actions on Riemannian manifolds. Several years later, the problem of classifying hyperpolar actions on symmetric spaces of compact type was posed in~\cite{HPTT95}. Hyperpolar actions on irreducible symmetric spaces of compact type have been classified by the third-named author in~\cite{k02}. The classification of polar actions on compact symmetric spaces of rank one was obtained by Podest\`{a} and Thorbergsson~\cite{PT99}. This classification shows that there are finitely many examples of polar, non-hyperpolar actions on each compact symmetric space of rank one.

After having completed the classification of polar actions on irreducible Hermitian symmetric spaces of compact type,  Biliotti~\cite{Bi06} formulated the following conjecture: \emph{a polar action on an irreducible symmetric space of compact type and higher rank is hyperpolar.}

The third author proved that the conjecture holds for symmetric spaces with simple isometry group~\cite{K07}, and for the exceptional simple Lie
groups~\cite{K09}. In recent work, Lytchak~\cite{lytchak} obtained a decomposition theorem for the more general case of singular polar foliations on nonnegatively curved, not necessarily irreducible, Riemannian symmetric spaces. In particular, his result shows that polar actions on irreducible symmetric spaces of higher rank are hyperpolar if the cohomogeneity is at least three.
Kollross and Lytchak~\cite{KL} then completed the proof that Biliotti's conjecture holds and hence, that the classification of polar actions on irreducible symmetric spaces of compact type follows from~\cite{k02} and~\cite{PT99}. Note that the classification of polar actions on reducible symmetric spaces cannot be obtained from the corresponding classification in irreducible ones. However, by the structural result~\cite[Theorem~5.5]{KrLy}, it now only remains to classify nondecomposable hyperpolar actions on reducible spaces in order to obtain a complete classification of polar actions on symmetric spaces of compact type up to orbit equivalence.

But while polar actions on nonnegatively curved symmetric spaces are almost completely classified, the situation in the noncompact case remains largely open. Wu~\cite{wu} classified polar actions on real hyperbolic spaces and showed that, up to orbit equivalence, they are products of a noncompact factor (which is either the isometry group of a lower dimensional real hyperbolic space or the nilpotent part of its Iwasawa decomposition) and a compact factor (which comes from the isotropy representation of a symmetric space). In particular, there are only finitely many examples of polar actions on a real hyperbolic space up to orbit equivalence. Berndt and the first-named author obtained in~\cite{BD11a} the classification of polar actions on the complex hyperbolic plane~$\C H^2$, showing that there are exactly nine examples up to orbit equivalence. No other complete classification of polar actions was previously known for the symmetric spaces of noncompact type. In this paper we present the classification of polar actions on complex hyperbolic spaces of arbitrary dimension. It is a remarkable consequence of Theorems~A and~B below that the cardinality of the set of polar actions on a complex hyperbolic space~$\C H^n$, $n\geq 3$, is infinite, and hence the methods used in~\cite{BD11a} cannot be applied in this more general situation.

An important fact to bear in mind here is that, in general, the duality of Riemannian symmetric spaces cannot be applied to derive classifications of polar actions on noncompact symmetric spaces from the corresponding classifications in the compact setting. Nevertheless, there are certain situations where duality can be used to obtain partial classifications. The first and the third authors derived in~\cite{DK11} the classification of polar actions with a fixed point on symmetric spaces using this method. They have shown that a polar action with a fixed point in a reducible symmetric space splits as a product of polar actions on each factor. The third author explored this idea further and obtained a classification of polar actions by algebraic reductive subgroups~\cite{k11}.

Berndt and Tamaru~\cite{BT07} classified cohomogeneity one actions on
complex hyperbolic spaces, the quaternionic hyperbolic plane, and the
Cayley hyperbolic plane. The classification remains open in
quaternionic hyperbolic spaces of higher dimension, where the first and second authors have recently obtained new examples of such actions~\cite{DD13}, and in noncompact symmetric
spaces of higher rank. See~\cite{BT} for more information on
cohomogeneity one actions on symmetric spaces of noncompact type.

A polar action on a symmetric space of compact
type always has singular orbits. Motivated by this fact, Berndt,
Tamaru and the first author studied hyperpolar actions on symmetric
spaces that have no singular orbits~\cite{BDT10} and obtained a
complete classification. It was also shown in this paper that there
are polar actions on symmetric spaces of noncompact type and rank
higher than one that are not hyperpolar, unlike in the compact
setting. This classification has been improved in complex hyperbolic
spaces, where Berndt and the first author classified polar
homogeneous foliations~\cite{BD11b}. The main result of this paper
contains~\cite{BD11a},~\cite{BD11b} and~\cite{BT07} as particular
cases.

Let $\C H^n=G/K$ be the complex hyperbolic $n$-space, where
$G=SU(1,n)$ and $K=S(U(1)U(n))$ is the isotropy group of $G$ at some
point $o$. Consider the Cartan decomposition $\g{g}=\g{k} \oplus
\g{p}$ with respect to $o$. Choose a maximal abelian subspace~$\g{a}$
of~$\g{p}$ and let $\g{g} =
\g{g}_{-2\alpha}\oplus\g{g}_{-\alpha}\oplus\g{g}_0\oplus\g{g}_{\alpha}\oplus\g{g}_{2\alpha}$
be the root space decomposition with respect to~$\g{a}$. Set $\g{k}_0
= \g{k} \cap \g{g}_0 \cong \g{u}(n-1)$. Since $\g{k}_0$ acts on the
root space~$\g{g}_\alpha$, the center of $\g{k}_0$ induces a natural
complex structure~$J$ on $\g{g}_\alpha$ which makes it isomorphic to
$\C^{n-1}$. On the other hand, we call a subset of~$\g{g}_\alpha$ a
\emph{real subspace} of $\g{g}_\alpha$ if it is a linear subspace
of~$\g{g}_{\alpha}$, where $\g{g}_{\alpha}$ is viewed as a real
vector space. Assume $\g{g}_\alpha$ is endowed with the inner product
given by the restriction of the Killing form of~$\g{g}$. A real
subspace~$\g{w}$ of $\g{g}_\alpha$ is said to be \emph{totally real}
if $\g{w} \perp J(\g{w})$.

In this paper, we prove the following classification result:

\begin{theoremA}
For each of the Lie algebras $\g{h}$ below, the corresponding
connected subgroup of $U(1,n)$ acts polarly on~$\C H^n$:

\begin{enumerate}[{\rm (i)}]

\item $\g{h} = \g{q} \oplus \g{so}(1,k)  \subset \g{u}(n-k)
    \oplus \g{su}(1,k)$, $k\in\{0,\dots,n\}$, where $\g{q}$ is a
    subalgebra of~$\g{u}(n-k)$ such that the corresponding
    subgroup~$Q$ of~$U(n-k)$ acts polarly with a totally real
    section on~$\C^{n-k}$.\label{th:main:so}

\item $\g{h} = \g{q} \oplus \g{b} \oplus \g{w} \oplus \g{g}_{2\alpha}\subset\g{su}(1,n)$,
    where $\g{b}$ is a linear subspace of~$\g{a}$, $\g{w}$ is a real
    subspace of~$\g{g}_{\alpha}$, and $\g{q}$ is a subalgebra of~$\g{k}_0$
    which normalizes~$\g{w}$ and such that the connected subgroup of $SU(1,n)$
    with Lie algebra $\g{q}$ acts polarly with a totally real section on
    the orthogonal complement of~$\g{w}$
    in~$\g{g}_{\alpha}$.\label{th:main:parabolic}
\end{enumerate}
Conversely, every nontrivial polar action on $\C H^n$ is orbit
equivalent to one of the actions above.
\end{theoremA}

We say that an action is trivial if it fixes all points. Transitive actions are not considered trivial in this paper. Taking for example $\g{q}=\g{k}_0$, $\g{b}=\g{a}$ and $\g{w}=\g{g}_\alpha$ in Theorem~A~\eqref{th:main:parabolic}, we obtain a transitive action on $\C H^n$.

In case~\eqref{th:main:so} of Theorem~A, one orbit of the $H$-action
is a totally geodesic $\R H^k$ and the other orbits are contained in
the distance tubes around it. In case~\eqref{th:main:parabolic}, if
$\g{b}=\g{a}$, one $H$-orbit of minimum orbit type contains a
geodesic line, while if $\g{b}=0$, any $H$-orbit of minimum orbit
type is contained in a horosphere.

We would like to remark here that Theorem~A actually provides many
examples of polar actions on~$\C H^n$. Indeed, for every choice of a
real subspace $\g{w}$ in~$\g{g}_\alpha$, there is at least one polar
action as described in part~(ii) of Theorem~A, see
Section~\ref{sec:examples}.

With the notation as in Theorem~A, we can determine the orbit equivalence
classes of the polar actions given in the theorem above.

\begin{theoremB}
Let $H_1$ and $H_2$ be two subgroups of $U(1,n)$ acting polarly on
$\C H^n$ as given by Theorem~A, and let $\g{h}_1$ and $\g{h}_2$ be
their corresponding Lie algebras. Then the actions of $H_1$ and $H_2$
are orbit equivalent if and only if one of the following conditions
holds:
\begin{enumerate}[{\rm (a)}]
\item $\g{h}_i=\g{q}_i\oplus\g{so}(1,k)$, $i\in\{1,2\}$, and the
    actions of $Q_1$ and $Q_2$ on $\C^{n-k}$ are orbit
    equivalent.

\item
    $\g{h}_i=\g{q}_i\oplus\g{b}_i\oplus\g{w}_i\oplus\g{g}_{2\alpha}$,
    $i\in\{1,2\}$, $\g{b}_1=\g{b}_2$, there exists an element
    $k\in K_0$ such that $\g{w}_2=\Ad(k)\g{w}_1$, and the actions
    of $Q_i$ on the orthogonal complement of $\g{w}_i$ in
    $\g{g}_\alpha$ are orbit equivalent for
    $i\in\{1,2\}$.\label{th:congruence:parabolic}
\end{enumerate}
\end{theoremB}

By means of the concept of
K\"{a}hler angle, we can give an equivalent way of characterizing the congruence of
subspaces of $\g{g}_\alpha$ by an element of $K_0$ stated in
Theorem~B(\ref{th:congruence:parabolic}). A subspace $\g{w}$ of $\g{g}_\alpha\cong\C^{n-1}$ is
said to have \emph{constant K\"{a}hler angle} $\varphi\in[0,\pi/2]$ if for each
nonzero vector $v\in\g{w}$ the angle between $Jv$ and $\g{w}$ is
precisely $\varphi$. In Subsection~\ref{subsec:realsubspace} we show
that any real subspace $\g{w}$ of $\g{g}_\alpha$ admits a
decomposition $\g{w}=\oplus_{\varphi\in\Phi}\g{w}_\varphi$ into
subspaces of constant K\"{a}hler angle, where $\Phi$ is the set of the
different K\"{a}hler angles arising in this decomposition, and
$\g{w}_\varphi$ has constant K\"{a}hler angle $\varphi$. This
decomposition is unique up to the ordering of the addends. Two
subspaces $\g{w}_1=\oplus_{\varphi\in\Phi_1}\g{w}_{1,\varphi}$ and
$\g{w}_2=\oplus_{\varphi\in\Phi_2}\g{w}_{2,\varphi}$ of
$\g{g}_\alpha$ are then congruent by an element of $K_0\cong U(n-1)$
if and only if $\Phi_1=\Phi_2$ and
$\dim\g{w}_{1,\varphi}=\dim\g{w}_{2,\varphi}$ for each~$\varphi$.

It follows in particular from Theorems~A and~B that the moduli space
of polar actions on $\C H^n$ up to orbit equivalence is finite if
$n=2$, cf.~\cite{BD11a}, and uncountable infinite in case $n \ge 3$.
Indeed, in dimension $n \ge 3$ the action of the group $U(n-1)$ on
the set of real subspaces of dimension $k$ of $\C^{n-1}$, with
$k\in\{2,\dots, 2n-4\}$, is not transitive. The orbits of this action
are determined by the decomposition of a real subspace into a sum of
spaces of constant K\"{a}hler angle, and the latter are parametrized by
the set $[0,\pi/2]$, which is uncountable infinite. As a
consequence, there are uncountably many polar, non-hyperpolar actions
on $\C H^n$, $n\geq 3$, up to orbit equivalence.

\medskip

This paper is organized as follows. In
Section~\ref{sec:preliminaries} we review the basic facts and
notations on complex hyperbolic spaces (\S\ref{subsec:CHn}), polar
actions (\S\ref{subsec:polar}), and real vector subspaces of complex
vector spaces (\S\ref{subsec:realsubspace}). The results of
Subsection~\ref{subsec:realsubspace} will be crucial for the rest of
the paper. Section~\ref{sec:examples} is devoted to present the new
examples that appear in Theorem~A. We also present here an outline of
the proof of Theorem~A. This proof has two main parts depending on
whether the group acting leaves a totally geodesic subspace
invariant (Section~\ref{sec:totally_geodesic}) or is contained in a
maximal parabolic subgroup of $SU(1,n)$
(Section~\ref{sec:parabolic}). We conclude in Section~\ref{sec:proof}
with the proofs of Theorems~A and~B.

%


\section{Preliminaries}\label{sec:preliminaries}

In this section we introduce the main known results and notation used
throughout this paper. We would like to emphasize the importance of
Subsection~\ref{subsec:realsubspace}, which is pivotal in the
construction and classification of new examples of polar actions on
complex hyperbolic spaces.

As a matter of notation, if $U_1$ and $U_2$ are two linear subspaces
of a vector space $V$, then $U_1\oplus U_2$ denotes their (not
necessarily orthogonal) direct sum. We will frequently use the
following notation for the orthogonal complement of a subspace of a
real vector space endowed with a scalar product, namely, by $V
\ominus U$ we denote the orthogonal complement of the linear
subspace~$U$ in the Euclidean vector space~$V$.

\subsection{The complex hyperbolic space}\label{subsec:CHn}
In this subsection we recall some well-known facts and notation on
the structure of the complex hyperbolic space as a symmetric space.
This will be fundamental for the rest of the work. As usual, Lie
algebras are written in gothic letters.

We will denote by $\C H^n$ the complex hyperbolic space with constant
holomorphic sectional curvature $-1$. As a symmetric space, $\C H^n$
is the coset space $G/K$, where $G=SU(1,n)$, and $K=S(U(1)U(n))$ is
the isotropy group at some point $o\in \C H^n$. Let
$\g{g}=\g{k}\oplus\g{p}$ be the Cartan decomposition of $\g{g}$ with
respect to $o$, where $\g{p}$ is the orthogonal complement of $\g{k}$
in $\g{g}$ with respect to the Killing form $B$ of $\g{g}$. Denote by
$\theta$ the corresponding Cartan involution, which satisfies
$\theta\rvert_\g{k}=\id$ and $\theta\rvert_\g{p}=-\id$. Note that the
orthogonal projections onto $\g{k}$ and $\g{p}$ are
$\frac{1}{2}(1+\theta)$ and $\frac{1}{2}(1-\theta)$, respectively.
Let $\ad$ and $\Ad$ be the adjoint maps of $\g{g}$ and $G$,
respectively. It turns out that $\langle X, Y\rangle=-B(\theta X, Y)$
defines a positive definite inner product on $\g{g}$ satisfying the
relation $\langle\ad(X)Y,Z\rangle=-\langle Y,\ad(\theta X)Y\rangle$
for all $X$, $Y$, $Z\in\g{g}$. Moreover, we can identify $\g{p}$ with
the tangent space $T_o\C H^n$ of $\C H^n$ at the point $o$.

Since $\C H^n$ has rank one, any maximal abelian subspace $\g{a}$  of
$\g{p}$ is $1$-dimensional. For each linear functional $\lambda$ on
$\g{a}$, define $\g{g}_\lambda=\{X\in\g{g}:\ad(H)X=\lambda(H)X\text{
for all } H\in\g{a}\}$.
Then $\g{a}$ induces the restricted root
space decomposition
$\g{g}=\g{g}_{-2\alpha}\oplus\g{g}_{-\alpha}\oplus\g{g}_0
\oplus\g{g}_\alpha\oplus\g{g}_{2\alpha}$,
which is an orthogonal direct sum with respect to
$\langle\cdot,\cdot\rangle$ satisfying
$[\g{g}_\lambda,\g{g}_\mu]=\g{g}_{\lambda+\mu}$ and $\theta
\g{g}_\lambda=\g{g}_{-\lambda}$. Moreover,
$\g{g}_0=\g{k}_0\oplus\g{a}$, where $\g{k}_0=\g{g}_0\cap
\g{k}\cong\g{u}(n-1)$ is the normalizer of $\g{a}$ in $\g{k}$. The
root space $\g{g}_\alpha$ has dimension $2n-2$, while
$\g{g}_{2\alpha}$ is $1$-dimensional, and both are normalized by
$\g{k}_0$.

We define $\g{n}=\g{g}_\alpha\oplus\g{g}_{2\alpha}$, which is a
nilpotent subalgebra of $\g{g}$ isomorphic to the
$(2n-1)$-dimensional Heisenberg algebra. The corresponding Iwasawa
decomposition of $\g{g}$ is $\g{g}=\g{k}\oplus\g{a}\oplus\g{n}$. The
connected subgroup of $G$ with Lie algebra $\g{a}\oplus\g{n}$ acts
simply transitively on $\C H^n$. One may endow $AN$, and then
$\g{a}\oplus\g{n}$, with the left-invariant metric
$\langle\cdot,\cdot\rangle_{AN}$ and the complex structure $J$ that
make $\C H^n$ and $AN$ isometric and isomorphic as K\"ahler
manifolds. Then $\langle X, Y\rangle_{AN}=\langle X_{\g{a}},
Y_\g{a}\rangle+\frac{1}{2}\langle X_\g{n}, Y_\g{n}\rangle$ for $X$,
$Y\in\g{a}\oplus\g{n}$; here subscripts mean the $\g{a}$ and $\g{n}$ components respectively.
The complex structure $J$ on $\g{a}\oplus\g{n}$ leaves $\g{g}_\alpha$
invariant, turning $\g{g}_\alpha$ into an $(n-1)$-dimensional complex
vector space $\C^{n-1}$. Moreover, $J\g{a}=\g{g}_{2\alpha}$.

Let $B\in\g{a}$ be a unit vector and define $Z=JB\in\g{g}_{2\alpha}$.
Then $\langle B, B\rangle=\langle B, B\rangle_{AN}=1$ and $\langle Z,
Z\rangle =2\langle Z, Z\rangle_{AN}=2$. The Lie bracket of
$\g{a}\oplus\g{n}$ is given by
\[
[aB+U+xZ,bB+V+yZ]=-\frac{b}{2}U+\frac{a}{2}V+\left(-bx+ay+
\frac{1}{2}\langle JU,V\rangle\right)Z,
\]
where $a$, $b$, $x$, $y\in\R$, and $U$, $V\in\g{g}_\alpha$. Let us
also  define $\g{p}_\lambda=(1-\theta)\g{g}_\lambda$, the projection
onto $\g{p}$ of the restricted root spaces. Then
$\g{p}=\g{a}\oplus\g{p}_\alpha\oplus\g{p}_{2\alpha}$. If the complex
structure on $\g{p}$ is denoted by $i$, then we have that
$2iB=(1-\theta)Z$, and $i(1-\theta)U=(1-\theta)JU$ for every
$U\in\g{g}_\alpha$.

We state now two lemmas that will be used frequently throughout the article.

\begin{lemma}\label{lemma:aux}
We have:
\begin{enumerate}[\rm (a)]
\item $[\theta X,Z]=-JX$ for each
    $X\in\g{g}_\alpha$.\label{lemma:aux:a}

\item $\langle T,(1+\theta)[\theta X,Y]\rangle=2\langle
    [T,X],Y\rangle$, for any $X$, $Y\in\g{g}_\alpha$ and
    $T\in\g{k}_0$.\label{lemma:aux:b}
\end{enumerate}
\end{lemma}

\begin{proof}
See \cite[Lemma~2.1]{BD11b}.
\end{proof}

\begin{lemma}\label{lemma:equivariance}
The orthogonal projection map
$\frac{1}{2}(1-\theta)\colon\g{a}\oplus\g{g}_\alpha\oplus\g{g}_{2\alpha}
\to\g{a}\oplus\g{p}_\alpha\oplus\g{p}_{2\alpha}$ defines an equivalence between
the adjoint $K_0$-representation on
$\g{a}\oplus\g{g}_\alpha\oplus\g{g}_{2\alpha}$ and the adjoint
$K_0$-representation on $\g{p}=\g{a}\oplus\g{p}_\alpha\oplus\g{p}_{2\alpha}$.
Moreover, this equivalence is an isometry between
$(\g{a}\oplus\g{g}_\alpha\oplus\g{g}_{2\alpha},\langle\cdot,\cdot\rangle_{AN})$
and $(\g{p},\langle\cdot,\cdot\rangle)$, and
$\frac{1}{2}(1-\theta)\colon\g{g}_\alpha\to\g{p}_\alpha$ is a complex linear
map.
\end{lemma}

\begin{proof}
The first part follows from the fact that $\theta$ is a
$K$-equivariant,  hence $K_0$-equivariant, map on~$\g{g}$. The other
claims follow from the facts stated above in this subsection.
\end{proof}

\subsection{Polar actions}\label{subsec:polar}

Let $M$ be a Riemannian manifold and $I(M)$ its isometry group. It is
known that $I(M)$ is a Lie group. Let $H$ be a connected closed
subgroup of $I(M)$. The action of $H$ on $M$ is called \emph{polar}
if there exists an immersed connected submanifold $\Sigma$ of $M$ such that:
\begin{enumerate}
\item $\Sigma$ intersects all the orbits of the $H$-action, and

\item for each $p\in\Sigma$, the tangent space of $\Sigma$ at
    $p$, $T_p\Sigma$, and the tangent space of the orbit through
    $p$ at $p$, $T_p(H\cdot p)$, are orthogonal.
\end{enumerate}
In such a case, the submanifold $\Sigma$ is called a \emph{section}
of the $H$-action. The action of $H$ is called \emph{hyperpolar} if
the section $\Sigma$ is flat in its induced Riemannian metric.

Two isometric Lie group actions on two Riemannian manifolds~$M$
and~$N$ are said to be \emph{orbit equivalent} if there is an
isometry $M \to N$ which maps connected components of orbits onto
connected components of orbits. They are said to be \emph{conjugate}
if there exists an equivariant isometry $M \to N$.

The final aim of our research is to classify polar actions on a given
Riemannian manifold up to orbit equivalence. In this paper we
accomplish this task for complex hyperbolic spaces. See the survey
articles~\cite{T05}, \cite{T10} and~\cite{D13} for more information and
references on polar actions.

Since $\C H^n$ is of rank one, a polar action on~$\C H^n$ is
hyperpolar if and only if it is of cohomogeneity one, i.e.\ the
orbits of maximal dimension are hypersurfaces. Conversely, any action
of cohomogeneity one on~$\C H^n$ (or any other Riemannian symmetric
space) is hyperpolar. Cohomogeneity one actions on complex hyperbolic
spaces have been classified by Berndt and Tamaru in~\cite{BT07}.

{}From now on we focus on polar actions on complex hyperbolic spaces
and recall or prove some facts that will be used later in this
article. We begin with a criterion that allows us to decide whether
an action is polar or not. The first such criterion of polarity is
credited to Gorodski~\cite{G04}.

\begin{proposition}\label{prop:criterion2}\label{prop:criterion1}
Let $M = G/K$ be a Riemannian symmetric space of noncompact type, and let
$\Sigma$ be a connected totally geodesic submanifold of $M$ with $o \in
\Sigma$. Let $H$ be a closed subgroup of~$I(M)$. Then $H$ acts polarly on $M$
with section $\Sigma$ if and only if $T_o\Sigma$ is a section of the slice
representation of $H_o$ on $\nu_o(H \cdot o)$, and $\langle
\g{h},T_o\Sigma\oplus[T_o\Sigma,T_o\Sigma]\rangle=0$.

In this case, the following conditions are satisfied:
\begin{itemize}
\item[(a)] $T_o\Sigma \oplus[\g{h}_o,\xi]=\nu_o(H\cdot o)$ for each regular
    normal vector $\xi\in\nu_o(H\cdot o)$.
\item[(b)] $T_o\Sigma \oplus[\g{h}_o,T_o\Sigma]=\nu_o(H\cdot o)$.
\item[(c)] $\Ad(H_o)T_o\Sigma=\nu_o(H\cdot o)$.
\end{itemize}
\end{proposition}
\begin{proof}
Follows from \cite[Corollary~3.2]{BD11a} and from well-known facts on polar
representations of compact groups \cite{Da85}.
\end{proof}

If $N$ is a submanifold of $\C H^n$, then $N$ is said to be \emph{totally
real} if for each $p\in N$ the tangent space $T_p N$ is a totally real
subspace of $T_p\C H^n$, that is, $J T_p N$ is orthogonal to $T_p N$.
See \S\ref{subsec:realsubspace} for more information of totally real
subspaces of complex vector spaces. The next theorem shows that
sections are necessarily totally real.

\begin{proposition}\label{prop:realSection}
Let $H$ act nontrivially, nontransitively, and polarly on the complex
hyperbolic space $\C H^n$, and let $\Sigma$ be a section of this
action. Then, $\Sigma$ is a totally real submanifold of $\C H^n$.
\end{proposition}

\begin{proof}
Since the action of $H$ is polar, the section $\Sigma$ is a totally
geodesic submanifold of $\C H^n$, hence $\Sigma$ is either totally
real or complex. Assume that $\Sigma$ is complex.

Since all sections are of the form $h(\Sigma)$, with $h\in H$, and the
isometries of $H$ are holomorphic, it follows that any principal orbit is
almost complex. It is a well-known fact that an almost complex submanifold in a
K\"{a}hler manifold is K\"{a}hler.

Since every $H$-equivariant
normal vector field on a principal orbit is parallel with respect to
the normal connection~\cite[Corollary 3.2.5]{BCO03}, then this
principal orbit is either a point or $\C H^n$ (see for
example~\cite{ADS04}), contradiction. Therefore $\Sigma$ is totally
real.
\end{proof}

Finally, we prove a general result that will help us to decide when two actions are orbit equivalent.
\begin{lemma}\label{lemma:sameslice}
Let $H\subset\tilde{H}$ be connected groups of isometries of a complete connected Riemannian manifold $M$ acting properly on $M$. Assume the following conditions are satisfied:
\begin{enumerate}[{\rm (i)}]
\item There exists $o\in M$ such that $H\cdot o=\tilde{H}\cdot o$.
\item The principal orbits of the slice representation of $H$ at $o$ are orbits of the slice representation of $\tilde{H}$ at $o$.
\end{enumerate}
Then $H$ and $\tilde{H}$ act on $M$ with the same orbits.
\end{lemma}
\begin{proof}
For any regular vector $\xi\in \nu_o(H\cdot o)=\nu_o(\tilde{H}\cdot o)$ of the slice representation of $H$ at $o$, the codimension of an orbit of $H$ (resp.\ of $\tilde{H}$) through $\exp_o(\xi)$ coincides with the codimension of the orbit of the slice representation of $H$ (resp.\ of $\tilde{H}$) through $\xi$. Since the orbits of $H$ are contained in the orbits of $\tilde{H}$, we have that the principal orbits of the $H$-action on $M$ are also orbits of the $\tilde{H}$-action on $M$.

Now let $H\cdot q$ be any orbit of the $H$-action on $M$, not necessarily principal. Fix a principal $H$-orbit $H\cdot p$ that is also a principal $\tilde{H}$-orbit; this is possible since the union of all principal orbits of the proper action of $H$ (and also of $\tilde{H}$) is an open dense subset of $M$. Let $\eta$ be the $H$-equivariant normal vector field along the principal orbit $H\cdot p$ whose exponentation gives the orbit $H\cdot q$, i.e.\ $H\cdot q=\{\exp_{h(p)}\eta:h\in H\}$, which exists by \cite[\S3.1h]{BCO03}. Since  the orbit of $\tilde{H}$ through $p$ is principal, the vector $\eta_p$ extends in a unique way to an $\tilde{H}$-equivariant normal vector field along $\tilde{H}\cdot p=H\cdot p$. This normal field is also $H$-equivariant since $H\subset\tilde{H}$, and thus it must coincide with $\eta$. Again by \cite[\S3.1h]{BCO03}, the exponentiation of $\eta$ produces the $\tilde{H}$-orbit through $q$, that is, $\tilde{H}\cdot q=\{\exp_{h(p)}\eta:h\in \tilde{H}\}$. But then $H\cdot q=\tilde{H}\cdot q$.
\end{proof}

\subsection{The structure of a real subspace of a complex vector
space}\label{subsec:realsubspace}

Let us denote by $J$ the complex structure of the complex vector
space $\C^n$. We view $\C^n$ as a Euclidean vector space with the
scalar product given by the real part of the standard Hermitian
scalar product. We define a \emph{real subspace} of~$\C^n$ to be an
$\R$-linear subspace of the real vector space obtained from $\C^n$ by
restricting the scalars to the real numbers. Let $V$ be a real
subspace of~$\C^n$. We will denote by  $\pi_V$ the orthogonal
projection map onto $V$.

The \emph{K\"ahler angle} of a nonzero vector $v\in V$ with respect
to $V$ is defined to be the angle between $Jv$ and $V$ or,
equivalently, the value $\varphi\in[0,\pi/2]$ such that $\langle
\pi_V Jv, \pi_V Jv\rangle=\cos^2(\varphi)\langle v, v\rangle$. We say
that $V$ has \emph{constant K\"ahler angle} $\varphi$ if the K\"ahler
angle of every nonzero vector $v\in V$ with respect to $V$ is
$\varphi$. In particular, $V$ is a \emph{complex subspace} if and
only if it has constant K\"{a}hler angle~$0$; it is a \emph{totally real
subspace} if and only if it has constant K\"{a}hler angle~$\pi/2$.

\begin{remark}\label{remark:KahlerAngle}
If $\{e_1,\dots,e_n\}$ and $\{f_1,\dots,f_n\}$ both are $\C$-orthonormal
bases of~$\C^n$, then the real subspace~$V_\varphi$ of
$\C^{2n}=\C^n\oplus\C^n$ generated by
\[
\{\cos({\textstyle\frac{\varphi}{2}})e_1+\sin({\textstyle\frac{\varphi}{2}})Jf_1,
\cos({\textstyle\frac{\varphi}{2}})Je_1+\sin({\textstyle\frac{\varphi}{2}})f_1,\dots,
\cos({\textstyle\frac{\varphi}{2}})e_n+\sin({\textstyle\frac{\varphi}{2}})Jf_n,
\cos({\textstyle\frac{\varphi}{2}})Je_n+\sin({\textstyle\frac{\varphi}{2}})f_n\}
\]
has constant K\"ahler angle $\varphi\in[0,\pi/2)$. Conversely, any
subspace of constant K\"ahler angle~$\varphi\in[0,\pi/2)$ and
dimension~$2n$ of~$\C^{2n}$ can be constructed in this way up to a transformation in $U(2n)$,
see~\cite{BB01}. In particular, it follows that two real subspaces of $\C^{n}$ with the same dimension and the same constant K\"{a}hler angle are congruent by an element of $U(n)$.
\end{remark}

For general real subspaces of a complex vector space, we have the
following structure result.

\begin{theorem}\label{th:structure}
Let $V$ be any real subspace of~$\C^n$. Then $V$ can be decomposed in a unique
way as an orthogonal sum of subspaces $V_i$, $i=1,\dots,r$, such that:
\begin{itemize}
\item[(a)] Each real subspace $V_i$ of~$\C^n$ has constant K\"ahler  angle
    $\varphi_i$.
\item[(b)] $\C V_i\perp\C V_j$, for every $i\neq j$, $i,j\in\{1,\dots,r\}$.
\item[(c)] $\varphi_1<\varphi_2<\dots<\varphi_r$.
\end{itemize}
\end{theorem}
\begin{proof}
The endomorphism $P=\pi_V\circ J$ of~$V$ is clearly skew-symmetric,
i.e. $\langle P v,w\rangle=-\langle v,P w\rangle$ for every $v$,
$w\in V$. Then, there exists an orthonormal basis of~$V$ for which
$P$ takes a block diagonal form with $2\times 2$ skew-symmetric
matrix blocks, and maybe one zero matrix block. Since $P$ is
skew-symmetric, its nonzero eigenvalues are pure imaginary. Assume then
that the distinct eigenvalues of~$P$ are $\pm i\lambda_1,\dots,\pm
i\lambda_r$ (maybe one of them is zero). We can and will further
assume that $\lvert \lambda_1\rvert>\dots>\lvert \lambda_r\rvert$.

Now consider the quadratic form $\Psi\colon V\to \R$ defined by
$\Psi(v)=\langle P v,P v\rangle=-\langle P^2 v,v\rangle$ for $v\in
V$. The matrix of this quadratic form $\Psi$ (or of the endomorphism
$-P^2$) with respect to the basis fixed above is diagonal with
entries $\lambda_1^2,\dots,\lambda_r^2$. For each $i=1,\dots,r$, let
$V_i$ be the eigenspace of~$-P^2$ corresponding to the eigenvalue
$\lambda_i^2$. Let $v\in V_i$ be a unit vector. Then
\[
\langle \pi_{V_i} J v, \pi_{V_i} Jv\rangle = \langle P v, \pi_{V_i} Jv\rangle
= \langle P v, P v\rangle = \Psi (v) =\lambda_i^2,
\]
where in the second and last equalities we have used that $P v\in
V_i$. This means that each subspace $V_i$ has constant K\"ahler angle
$\varphi_i$, where $\varphi_i$ is the unique value in $[0,\frac\pi2]$
such that $\lambda_i^2=\cos^2(\varphi_i)$.

By construction, it is clear that $V_i\perp V_j$ and $J V_i\perp J
V_j$  for $i\neq j$. Since for every $v\in V_i$ and $w\in V_j$,
$i\neq j$, we have that $\langle Jv,w\rangle=\langle P v,w\rangle=0$,
we also get that $J V_i\perp V_j$ if $i\neq j$. Hence $\C V_i\perp \C
V_j$ if $i\neq j$.

Property (c) follows from the assumption that $\lvert
\lambda_1\rvert>\dots>\lvert \lambda_r\rvert$, and this also implies the
uniqueness of the decomposition.
\end{proof}

It is convenient to change the notation of Theorem~\ref{th:structure}
slightly. Let $V$ be any real subspace of~$\C^n$, and let
$V=\bigoplus_{\varphi\in\Phi} V_{\varphi}$ be the decomposition
stated in Theorem~\ref{th:structure}, where $V_\varphi$ has constant
K\"{a}hler angle $\varphi\in[0,\pi/2]$, and $\Phi$ is the set of all
K\"{a}hler angles arising in this decomposition. Note that
according to Theorem~\ref{th:structure}, this decomposition is unique
up to the order of the factors. We agree to write $V_\varphi = 0$ if
$\varphi\notin\Phi$. The subspaces $V_0$ and $V_{\pi/2}$ (which can
be zero) play a somewhat distinguished role in the calculations that
follow, so we will denote $\Phi^*=\{\varphi\in\Phi:\varphi\neq
0,\pi/2\}$. Then, the above decomposition is written as
\begin{equation*}\label{eq:decomposition}
V=V_0\oplus\left(\bigoplus_{\varphi\in\Phi^*}V_\varphi\right)\oplus V_{\pi/2}.
\end{equation*}

For each $\varphi\in\Phi^*\cup\{0\}$, we define $J_\varphi\colon
V_{\varphi}\to V_{\varphi}$ by
$J_\varphi=\frac{1}{\cos(\varphi)}(\pi_{V_{\varphi}}\circ J)$. This
is clearly a skew-symmetric and orthogonal endomorphism
of~$V_{\varphi}$ (see the proof of Theorem~\ref{th:structure}).
Therefore $(V_{\varphi},J_\varphi)$ is a complex vector space for
every $\varphi\in\Phi^*\cup\{0\}$. Note that $J_0=J\vert_{V_0}$. Let
$U(V_{\varphi})$ be the group of all unitary transformations of the
complex vector space~$(V_{\varphi},J_\varphi)$, that is, the linear transformations of $V_\varphi$ that leave the Hermitian product $(v,w):=\langle v,w\rangle +
i\langle v,J_\varphi w\rangle$ invariant, or equivalently,
\[
U(V_\varphi)=\{A\in {GL}_\R(V_\varphi):A J_\varphi=J_\varphi A,
\langle Av,Aw\rangle=\langle v,w\rangle,\forall v,w\in V_\varphi\}.
\]

\begin{lemma}\label{lemma:perpphi}
Let $V$ be a real subspace of constant K\"{a}hler angle~$\varphi\neq0$ in
$\C^n$. Then the real subspace $\C V \ominus V$ of $\C^n$ has the
same  dimension as $V$ and constant K\"{a}hler angle~$\varphi$.
\end{lemma}

\begin{proof}
See for example~\cite[page~135]{BD09}.
\end{proof}

Let $V^\perp=\C^n\ominus V$, where as usual $\ominus$ denotes the
orthogonal complement. Then, Lemma~\ref{lemma:perpphi} implies that
the decomposition stated in Theorem~\ref{th:structure} can be written
as
\[
V^\perp=V_0^\perp\oplus\left(\bigoplus_{\varphi\in\Phi^*}
V_{\varphi}^\perp\right)\oplus V_{\pi/2}^\perp,\quad \text{where $\C
V_{\varphi}=V_\varphi\oplus V^\perp_\varphi$ for each
$\varphi\in\Phi^*\cup\{\pi/2\}$}.
\]
We define $m_{\varphi}=\dim V_\varphi$ and $m^\perp_{\varphi}=\dim
V^\perp_{\varphi}$. For every $\varphi\neq 0$ we have
$m_\varphi=m^\perp_\varphi$ by Lemma~\ref{lemma:perpphi}, but $V_0$
and $V_0^\perp$ are both complex subspaces of~$\C^n$, possibly of
different dimension.

\begin{lemma}\label{lemma:normalizer}
Let $V$ be a real subspace of~$\C^n$. Let $U(n)_V$ be the subgroup
of~$U(n)$ consisting of all the elements $A \in U(n)$ such that $A V
= V$. Then, we have the canonical isomorphism
\[
U(n)_V \cong \left[ \prod_{\varphi\in\Phi^*\cup\{0\}} U(V_{\varphi}) \right] \times
O(V_{\pi/2}) \times U(V_0^\perp),
\]
where we assume that $V_{\varphi}$, $\varphi \in\Phi^*\cup\{0\}$, is
endowed with the complex structure given by $J_\varphi =
\frac{1}{\cos(\varphi)}(\pi_{V_\varphi} \circ J)$, and that
$V_0^\perp$ is endowed with the complex structure given by the
restriction of~$J$.
\end{lemma}

\begin{proof}
Let $A \in U(n)$ be such that $A V = V$. Then $A$ commutes with $J$
and $\pi_V$ and hence leaves the eigenspaces of~$-P^2$ invariant (see
the proof of Theorem~\ref{th:structure}). Thus $A V_{\varphi} =
V_{\varphi}$. Since we also have $A V^\perp = V^\perp$, it follows
that $A V_{\varphi}^\perp = V_{\varphi}^\perp$.

Let $\varphi\in\Phi^*\cup\{0\}$. Since $AV_{\varphi}=V_{\varphi}$ and
$AV_{\varphi}^\perp=V_{\varphi}^\perp$ we have $A\C V_{\varphi}=\C
V_{\varphi}$. Clearly,
$A\circ\pi_{V_{\varphi}}\vert_{V_{\varphi}}=\pi_{V_{\varphi}}\circ
A\vert_{V_{\varphi}}$, and $A\circ\pi_{V_{\varphi}}\vert_{\C^n\ominus
V_{\varphi}}=0=\pi_{V_{\varphi}}\circ A\vert_{\C^n\ominus
V_{\varphi}}$. Hence, $A\circ\pi_{V_{\varphi}}=\pi_{V_{\varphi}}\circ
A$. Since $AJ=JA$ as well, we have that $A\circ
J_\varphi\vert_{V_{\varphi}}=J_\varphi\circ A\vert_{V_{\varphi}}$ on
$V_{\varphi}$, and thus, $A|_{V_{\varphi}} \in U(V_{\varphi})$. If
$\varphi = {\pi}/{2}$ then we have $AV_{{\pi}/{2}}=V_{{\pi}/{2}}$,
and clearly, $A\vert_{V_{\pi/2}}$ is an orthogonal transformation
of~$V_{\pi/2}$. Moreover, we have $A|_{V_0^\perp} \in U(V_0^\perp)$.
We define a map
\[
F\colon U(n)_V \to \left[ \prod_{\varphi\in\Phi^*\cup\{0\}} U(V_{\varphi}) \right] \times
O(V_{\pi/2}) \times U(V_0^\perp)
\]
by requiring that the projection onto each factor is given by the
corresponding restriction, that is, the $U(V_\varphi)$-projection of
$F(A)$ is given by $A\vert_{V_\varphi}$, the
$O(V_{{\pi}/{2}})$-projection of $F(A)$ is $A\vert_{V_{{\pi}/{2}}}$,
and the $U(V_0^\perp)$-projection of $F(A)$ is $A\vert_{V_0^\perp}$.

Since every element in $U(n)_V$ leaves the subspaces $V_{\varphi}$,
$\varphi\in\Phi$, and $V_0^\perp$ invariant, the map thus defined is
a homomorphism. Let us show injectivity and surjectivity. Let
$A_\varphi \in U(V_{\varphi})$ for each $\varphi\in\Phi^*\cup\{0\}$,
let $A_{{\pi}/{2}} \in O(V_{{\pi}/{2}})$, and let $A_0^\perp \in
U(V_0^\perp)$. If $A \in U(n)_V$ and $v \in JV_{\varphi}$ for
$\varphi\in\Phi$, then $Av$ is determined by $A_\varphi$ and $v$,
since $Av = - A J^2 v = - J A J v = - J A_\varphi (J v)$. Since we
have the direct sum decomposition
\[
\C^n = \left[\bigoplus_{\varphi\in\Phi}\C V_{\varphi}\right] \oplus V_0^\perp,
\]
it follows that the unitary map~$A$ on $\C^n$ is uniquely determined
by the maps $A_\varphi$, $\varphi\in\Phi$, and $A_0^\perp$. This
shows injectivity.

Conversely, let $A\in \left[ \prod_{\varphi\in\Phi^*\cup\{0\}}
U(V_{\varphi}) \right] \times O(V_{{\pi}/{2}}) \times U(V_0^\perp)$,
and denote by $A_\varphi$ the $U(V_\varphi)$-projection, by
$A_{{\pi}/{2}}$ the $O(V_{{\pi}/{2}})$-projection, and by $A_0^\perp$
the $U(V_0^\perp)$-projection. Then, we may construct a map $A \in
U(n)_V$ be defining $A(v+Jw)=A_\varphi v+JA_\varphi w$ for all $v$,
$w\in V_{\varphi}$, $\varphi\in\Phi$, $Av = A_0^\perp v$ for $v \in
V_0^\perp$, and extending linearly. For the map~$A$ thus defined we
have $A\vert_{V_{\varphi}} = A_\varphi$ for $\varphi\in\Phi$, and
$A\vert_{V_0^\perp} = A_0^\perp$. This proves surjectivity.
\end{proof}

\begin{remark}\label{rem:KahlerAngles}
Let $V$ and $W$ be two real subspaces of $\C^n$ whose K\"{a}hler angle decompositions have the same set of K\"{a}hler angles $\Phi$ and the same dimensions, that is, the decompositions given by Theorem~\ref{th:structure} are $V=\oplus_{\varphi\in\Phi}V_\varphi$ and $W=\oplus_{\varphi\in\Phi}W_\varphi$, with $\dim V_\varphi=\dim W_\varphi$ for all $\varphi\in\Phi$. Then, it follows from the results of this subsection that there exists a unitary automorphism $A$ of $\C^n$ such that $AV_\varphi=W_\varphi$ for all $\varphi\in\Phi$, and in particular, $AV=W$.
\end{remark}


\section{New examples of polar actions}\label{sec:examples}

We will now construct new examples of polar actions on complex
hyperbolic spaces. We will use the notation from
Subsection~\ref{subsec:CHn}.

Recall that the root space $\g{g}_\alpha$ is a complex vector space,
which we will identify with~$\C^{n-1}$. Let $\g{w}$ be a real
subspace of~$\g{g}_\alpha$ and
\[
\g{w}=\bigoplus_{\varphi\in\Phi} \g{w}_{\varphi}=
\g{w}_0\oplus\left(\bigoplus_{\varphi\in\Phi^*}\g{w}_\varphi\right)\oplus\g{w}_{\pi/2}
\]
its decomposition as
in Theorem~\ref{th:structure}, where $\Phi$ is the set of all
possible K\"{a}hler angles of vectors in $\g{w}$, $\Phi^*=\{\varphi\in\Phi:\varphi\neq
0,\pi/2\}$, and $\g{w}_\varphi$ has constant K\"{a}hler angle
$\varphi\in[0,{\pi}/{2}]$. Similarly, define
$\g{w}^\perp=\g{g}_\alpha\ominus\g{w}$ and let
\[
\g{w}^\perp=
\g{w}_0^\perp\oplus\left(\bigoplus_{\varphi\in\Phi^*}\g{w}_\varphi^\perp\right)
\oplus\g{w}_{\pi/2}^\perp
\]
be the corresponding decomposition as in Theorem~\ref{th:structure}.
We define $m_{\varphi}=\dim\g{w}_\varphi$ and
$m^\perp_{\varphi}=\dim\g{w}^\perp_\varphi$, and recall that
$m_\varphi=m_\varphi^\perp$ if $\varphi\in(0,\pi/2]$. Recall also
that $K_0$, the connected subgroup of $G=SU(1,n)$ with Lie algebra
$\g{k}_0$, is isomorphic to $U(n-1)$ and acts on
$\g{g}_\alpha\cong\C^{n-1}$ in the standard way. We denote by $N_{K_0}(\g{w})$ the normalizer of~$\g{w}$ in~$K_0$, and by $\g{n}_{\g{k}_0}(\g{w})$ the normalizer of~$\g{w}$ in~$\g{k}_0$.
We know from Lemma~\ref{lemma:normalizer} that
\begin{equation}\label{eq:wnormalizer}
N_{K_0}(\g{w})\cong\left[ \prod_{\varphi\in\Phi^*\cup\{0\}}
U(\g{w}_{\varphi}) \right] \times
O(\g{w}_{\pi/2}) \times U(\g{w}^\perp_{0}).
\end{equation}
This group leaves invariant each $\g{w}_\varphi$ and each
$\g{w}_\varphi^\perp$, and acts transitively on the unit spheres of these
subspaces of constant K\"ahler angle. Moreover, it acts polarly
on~$\g{w}^\perp$, see Remark~\ref{rem:k0polar} below.

The following result provides a large family of new examples of polar
actions on $\C H^n$.

\begin{theorem} \label{th:with_a_general}
Let $\g{w}$ be a real subspace of~$\g{g}_\alpha$ and $\g{b}$ a subspace
of $\g{a}$. Let
$\g{h}=\g{q}\oplus\g{b}\oplus\g{w}\oplus\g{g}_{2\alpha}$, where
$\g{q}$ is any Lie subalgebra of $\g{n}_{\g{k}_0}(\g{w})$ such that
the corresponding connected subgroup $Q$ of~$K$ acts polarly on
$\g{w}^\perp$ with section $\g{s}$. Assume $\g{s}$ is a totally real
subspace of $\g{g}_\alpha$. Then the connected subgroup $H$ of~$G$
with Lie algebra $\g{h}$ acts polarly on $\C H^n$ with section
$\Sigma=\exp_o((\g{a}\ominus\g{b})\oplus(1-\theta)\g{s})$.
\end{theorem}
\begin{proof}
We have that $T_o\Sigma=(\g{a}\ominus\g{b})\oplus(1-\theta)\g{s}$ and
$\nu_o(H\cdot o)= (\g{a}\ominus\g{b})\oplus(1-\theta)\g{w}^\perp$.
Since $\g{s}\subset\g{w}^\perp$, it follows that
$T_o\Sigma\subset\nu_o(H\cdot o)$. The slice representation of~$H_o$
on $\nu_o(H\cdot o)$ leaves the subspaces $\g{a}\ominus\g{b}$ and
$(1-\theta)\g{w}^\perp$ invariant. For the first one the action is
trivial, while for the second one the action is equivalent to the
representation of~$Q$ on $\g{w}^\perp$ (see
Lemma~\ref{lemma:equivariance}), which is polar with section $\g{s}$.
Hence, the slice representation of~$H_o$ on $\nu_o(H\cdot o)$ is
polar and $T_o\Sigma$ is a section of it. Let $v$,
$w\in\g{s}\subset\g{w}^\perp$. We have:
\[
[(1-\theta)v,(1-\theta)w] = (1+\theta)[v,w]-(1+\theta)[\theta v,w]
=-(1+\theta)[\theta v,w].
\]
The last equality holds because $v$ and $w$ lie in $\g{s}$, which is
a totally real subspace of~$\g{g}_\alpha$, and then
$[v,w]=\frac{1}{2}\langle Jv,w\rangle Z=0$. Since $v$,
$w\in\g{g}_\alpha$, then $\theta v\in\g{g}_{-\alpha}$ and $[\theta
v,w]\in\g{g}_0$. Hence $-(1+\theta)[\theta v,w]\in\g{k}_0$. Let
$X=T+aB+U+xZ\in\g{h}$, where $T \in \g{q}$, $U \in \g{w}$ and $a$, $x
\in \R$. Since $\g{k}_0$ is orthogonal
to~$\g{a}\oplus\g{g}_\alpha\oplus\g{g}_{2\alpha}$, we have:
\[
\langle[(1-\theta)v,(1-\theta)w],X\rangle=-\langle(1+\theta)[\theta v,w],T\rangle
=-2\langle[T,v],w\rangle= -4\langle[T,v],w\rangle_{AN} =0,
\]
where in the last equality we have used that the action of~$Q$ on
$\g{w}^\perp$ is a polar representation with section $\g{s}$. If
$\g{b}=\g{a}$, the result then follows using the criterion in
Proposition~\ref{prop:criterion1}.

If $\g{b}\neq\g{a}$ then $\g{b}=0$. In this case, let $v\in\g{s}$
and $X=T+U+xZ\in\g{h}$, where $T \in \g{q}$, $U \in \g{w}$, $x \in
\R$. Then:
\[
\langle[B,(1-\theta)v],X \rangle=\langle (1+\theta)[B,v],X\rangle
=\frac{1}{2}\langle (1+\theta)v,U\rangle=0.
\]
Since $[B,B]=0$, by linearity and the skew-symmetry of the Lie
bracket, it follows that $\langle
[T_o\Sigma,T_o\Sigma],\g{h}\rangle=0$. Again by
Proposition~\ref{prop:criterion1}, the result follows also in case
$\g{b}\neq\g{a}$.
\end{proof}

\begin{remark}\label{rem:k0polar}
In the special case $Q = N_{K_0}(\g{w})$, we obtain a polar action
on~$\C H^n$, since the whole normalizer $N_{K_0}(\g{w})$ acts polarly
on~$\g{w}^\perp$. Indeed, let $\g{s}_{\varphi}$ be any
one-dimensional subspace of $\g{w}^\perp_{\varphi}$ if
$\g{w}^\perp_{\varphi}\neq 0$, and define
$\g{s}=\bigoplus_{\varphi\in\Phi\cup\{0\}}\g{s}_{\varphi}$. Then
$\g{s}$ is a section of the action of~$N_{K_0}(\g{w})$
on~$\g{w}^\perp$. The cohomogeneity one examples introduced in
\cite{BB01} correspond to the case where $\g{w}^\perp$ has constant
K\"{a}hler angle, $\g{b} = \g{a}$ and $Q = N_{K_0}(\g{w})$.
\end{remark}

\begin{remark}
It is straightforward to describe all polar actions of closed
subgroups~$Q$ in Theorem~\ref{th:with_a_general} up to orbit
equivalence. In fact, the action of the group $N_{K_0}(\g{w})$ is
given by the products of the natural representations of the direct
factors in~\eqref{eq:wnormalizer} on the spaces
$\g{w}_{\varphi}^\perp$. By the main result of Dadok~\cite{Da85}, a
representation is polar if and only if it is orbit equivalent to the
isotropy representation of some Riemannian symmetric space.
Therefore, we obtain a representative for each orbit equivalence
class of polar actions on~$\g{w}^\perp$ given by closed subgroups
of~$N_{K_0}(\g{w})$ in the following manner. Given $\g{w}$, for each
$\varphi\in\Phi\cup\{0\}$ choose a Riemannian symmetric space $M_\varphi$
such that $\dim M_\varphi = \dim \g{w}_{\varphi}^\perp$. In case
${\pi}/{2}\in\Phi$, choose the symmetric spaces such that all of them
except possibly~$M_{{\pi}/{2}}$ are Hermitian symmetric; in case
${\pi}/{2}\notin\Phi$, choose all these symmetric spaces to be
Hermitian without exception. Then the isotropy representation of
$\prod_{\varphi\in\Phi\cup\{0\}}M_\varphi$ defines a closed subgroup
of~$N_{K_0}(\g{w})$, which acts polarly on $\g{w}^\perp$ with a
section~$\g{s}$, which is a totally real subspace of~$\g{g}_\alpha$,
see~\cite{PT99}. This construction exhausts all orbit equivalent
classes of closed subgroups in~$K_0$ leaving $\g{w}$ invariant and
acting polarly on~ $\g{w}^\perp$ with totally real section.
\end{remark}

\begin{remark}
There is a curious relation between some of the new examples of
polar actions in Theorem~\ref{th:with_a_general} and certain
isoparametric hypersurfaces constructed by the first two authors
in~\cite{DD12}. The orbit $H\cdot o$ of any of the polar actions
described in Theorem~\ref{th:with_a_general} with $\g{b}=\g{a}$ is
always a minimal (even austere) submanifold of $\C H^n$ that
satisfies the following property: the distance tubes around it are
isoparametric hypersurfaces which are hence foliated by orbits of the
$H$-action. Moreover, these hypersurfaces have constant principal
curvatures if and only if they are homogeneous (i.e.\ they are the
principal orbits of the cohomogeneity one action resulting from
choosing $\g{q}=\g{n}_{\g{k}_0}(\g{w})$ in
Theorem~\ref{th:with_a_general}); this happens precisely when the
real subspace $\g{w}^\perp$ of $\g{g}_\alpha$ has constant K\"ahler
angle. See \cite{DD12} for more details.
\end{remark}

The rest of the paper will be devoted to the proof of the classification
result stated in Theorems~A and~B. In order to justify the content of the
following sections, we will give here a sketch of the proof of
Theorem~A, and leave the details for the following sections.

Assume that $H$ is a closed subgroup of $SU(1,n)$ that acts polarly
on $\C H^n$. We will see that $H$ leaves a totally geodesic proper
subspace of $\C H^n$ invariant or it is contained in a maximal parabolic subgroup of $G$. In the
first case, $H$ leaves invariant a lower dimensional complex
hyperbolic space $\C H^k$, $k\in\{0,\dots,n-1\}$, or a real
hyperbolic space $\R H^n$. The first possibility is tackled in
Subsection~\ref{subsec:CHnInvariant}, and it follows from this part
of the paper that, roughly, the action of $H$ splits, up to orbit
equivalence, as the product of a polar action on the totally geodesic
$\C H^k$, and a polar action with a fixed point on its normal space.
Hence, the problem is reduced to the classification of polar actions
on lower dimensional complex hyperbolic spaces, which will allow us
to use an induction argument. The second possibility is addressed in
Subsection~\ref{subsec:RHnInvariant} where we show that the action of
$H$ is orbit equivalent to the action of the connected component of the identity of $SO(1,n)$, which is a
cohomogeneity one action whose orbits are tubes around a totally
geodesic $\R H^n$. Finally, assume that $H$ is contained in a maximal parabolic subgroup. The Lie algebra of a maximal parabolic subgroup
is of the form
$\g{l}=\g{k}_0\oplus\g{a}\oplus\g{g}_\alpha\oplus\g{g}_{2\alpha}$,
for some root space decomposition of $\g{su}(1,n)$
(see~\S\ref{subsec:CHn}). We show in Section~\ref{sec:parabolic} that
the Lie algebra of $H$ (up to orbit equivalence) must be of the form
$\g{q}\oplus\g{b}\oplus\g{w}\oplus\g{g}_{2\alpha}$, with
$\g{q}\subset\g{k}_0$, $\g{b}\subset\g{a}$, and
$\g{w}\subset\g{g}_\alpha$, or of the form $\g{q}\oplus\g{a}$, with
$\g{q}\subset\g{k}_0$. A bit more work leads us to the examples
described in Theorem~\ref{th:with_a_general}. Combining the different
cases, we will conclude in Section~\ref{sec:proof} the proofs of
Theorems~A and~B.


\section{Actions leaving a totally geodesic subspace invariant}\label{sec:totally_geodesic}

The results in this section show that in order to classify polar actions
leaving a totally geodesic complex hyperbolic subspace invariant it suffices to
study polar actions on the complex hyperbolic spaces of lower dimensions. We
will also show that actions leaving a totally geodesic $\R H^n$ invariant are orbit equivalent to the cohomogeneity one action of $SO^0(1,n)$, the connected component of the identity of $SO(1,n)$. Note that if an isometric action leaves a totally geodesic $\R H^k$ invariant, it also leaves a totally geodesic $\C H^k$ invariant.

The following is well known. Let $H$ be closed connected subgroup of~$SU(1,n)$.
If the natural action of~$H$ on $\C H^n$ leaves a totally geodesic proper
submanifold of~$\C H^n$ invariant, then there is an element $g \in SU(1,n)$
such that $gHg^{-1}$ is contained in one of the subgroups $S(U(1,k)  U(n-k))$
or $SO^0(1,n)$ of~$SU(1,n)$.

\subsection{Actions leaving a totally geodesic complex hyperbolic space
invariant}\label{subsec:CHnInvariant}

Let $L = S(U(1,k)  U(n-k)) \subset G= SU(1,n)$, $k\in\{1,\dots,n-1\}$. Let $M_1$ be the totally
geodesic $\C H^k$ given by the orbit $L \cdot o$. Let $M_2$ be the totally
geodesic $\C H^{n-k}$ which is the image of the normal space $\nu_o M_1$ under
the Riemannian exponential map $\exp_o$. Let $H$ be a closed connected subgroup
of~$L$. Then the $H$-action on~$\C H^n$ leaves $M_1$ invariant and the
$H$-action on~$\C H^n$ restricted to the isotropy subgroup $H_o$ leaves $M_2$
invariant. Let $\pi_1 \colon L \to U(1,k)$ and $\pi_2 \colon L \to U(n-k)$ be
the natural projections.

\begin{theorem}\label{theorem:ProdAct}
Assume the $H$-action on~$\C H^n$ is nontrivial. Then it is polar if
and only if the following hold:
\begin{enumerate}[{\rm (i)}]

  \item The action of~$H$ on~$M_1$ is polar and nontrivial.

  \item The action of $H_o$ on~$M_2$ is polar and nontrivial.

  \item The action of $\pi_1(H) \times \pi_2(H_o)$ on~$\C H^n$ is
      orbit equivalent to the $H$-action.

\end{enumerate}
\end{theorem}

\begin{proof}
Assume first that the $H$-action on~$\C H^n$ is polar and $\Sigma$ is
a section through $o$. Let $\Sigma_i$ be the connected component of $\Sigma \cap
M_i$ containing~$o$ for $i=1,2$. Obviously, the $H$-orbits on $M_1$
intersect $\Sigma_1$ orthogonally. Let $p$ be an arbitrary point
in~$M_1$. Then the intersection of the orbit $H \cdot p$
with~$\Sigma$ is non-empty. Let $q \in (H \cdot p) \cap \Sigma$.
Since $H$ leaves~$M_1$ invariant, we have that $q \in M_1$. Both the
Riemannian exponential maps of~$M_1$ and of~$\Sigma$ at the point~$o$
are diffeomorphisms by the Cartan-Hadamard theorem. Hence there is a
unique shortest geodesic segment~$\beta$ in~$\Sigma$ connecting $o$
with~$q$ and there is also a unique shortest geodesic
segment~$\gamma$ in~$M_1$ connecting $o$ with~$q$. Since both
$\Sigma$ and $M_1$ are totally geodesic submanifolds of~$\C H^n$ it
follows that $\beta$ and $\gamma$ are both geodesic
segments of~$\C H^n$ connecting the points $o$ and~$q$ and must
coincide by the Cartan-Hadamard theorem. Hence $\beta = \gamma$ both
lie in $\Sigma_1$. This shows that $\Sigma_1$ meets the $H$-orbit
through~$p$ (namely, at the point~$q$) and that the action of $H$
on~$M_1$ is polar.

Obviously, the $H_o$-orbits on~$M_2$ intersect $\Sigma_2$ orthogonally. Since
$T_oM_2$ is a submodule of the slice representation of~$H_o$ on~$\nu_o (H \cdot
o)$, the linear $H_o$-action on $T_oM_2$ is polar with section~$T_o\Sigma_2$.
The map $\exp_o \colon T_oM_2 \to M_2$ is an $H_o$-equivariant diffeomorphism
by the Cartan-Hadamard theorem. In particular, it follows that $\Sigma_2$ meets
all $H_o$-orbits in~$M_2$, since $T_o\Sigma_2$ meets all $H_o$-orbits in
$T_oM_2$. Thus the action of $H_o$ on~$M_2$ is polar.

Consider the polar slice representation of $H_o$ at $T_o\C H^n$ with
section~$T_o\Sigma$. By \cite[Theorem~4]{Da85}, it follows that
$T_o\Sigma = T_o\Sigma_1 \oplus T_o\Sigma_2$. Note that by Proposition~\ref{prop:realSection}, $T_o\Sigma$ is totally real, which implies that the actions of $H$ on $M_1$ and of $H_o$ on $M_2$ are nontrivial. We have shown that (i) and (ii) hold.

Let now $\g{g}=\g{k}\oplus\g{p}$ be the Cartan decomposition of $\g{g}=\g{su}(1,n)$ with respect to the point~$o$. We know that~$\g{p}$ can be identified with the tangent space $T_o\C H^n$. We consider the following Lie algebras: $\g{h}$ the Lie algebra of $H$, $\hat{\g{h}}=\pi_1(\g{h})\oplus\pi_2(\g{h}_o)$, and $\tilde{\g{h}}=\pi_1(\g{h})\oplus\pi_2(\g{h})$, where we have also denoted by $\pi_i$, $i=1,2$, the projections onto $\g{u}(1,k)$ and $\g{u}(n-k)$, respectively. Note that $\g{h}_o=\g{h}\cap\g{k}$. In what follows we show that the connected subgroups $H$, $\hat{H}$ and $\tilde{H}$ of $U(1,k)\times  U(n-k)$ whose Lie algebras are $\g{h}$, $\hat{\g{h}}$ and $\tilde{\g{h}}$ respectively, have the same orbits in $\C H^n$.

First observe that $\tilde{H}\cdot o=\hat{H}\cdot o=H\cdot o$, since $\pi_2(H)\subset K$.

\begin{claim*} 
For any regular vector $\xi\in T_o\Sigma$ of the slice representation of $H$ at $o$, we have $[\tilde{\g{h}}_o,\xi]=[\hat{\g{h}}_o,\xi]=[\g{h}_o,\xi]$, where $\tilde{\g{h}}_o=\tilde{\g{h}}\cap\g{k}$ and $\hat{\g{h}}_o=\hat{\g{h}}\cap\g{k}$.
\end{claim*}

Since the orbits through $o$ of the groups $H$, $\tilde{H}$ and $\hat{H}$ are the same, their isotropy groups at $o$ stabilize $\nu_o(H\cdot o)$. This already implies $[\tilde{\g{h}}_o,\xi]\subset\nu_o(H\cdot o)$, $[\hat{\g{h}}_o,\xi]\subset\nu_o(H\cdot o)$ and $[\g{h}_o,\xi]\subset\nu_o(H\cdot o)$.

Note that $[\g{h}_o,\xi]$ is precisely the tangent space of the orbit through $\xi$ of the slice representation of $H$ at $o$. Since $\xi$ is regular (that is, it lies in a principal orbit of the slice representation of $H$ at $o$), it follows that $[\g{h}_o,\xi]=\nu_o(H\cdot o)\ominus T_o\Sigma$.

Let $X\in\tilde{\g{h}}_o$ and $\eta\in T_o\Sigma$ be arbitrary. We have to show that $\langle [X,\xi],\eta\rangle=0$. We write $X=X_1+X_2$, with $X_i\in\pi_i(\g{h})$ and also $\xi=\xi_1+\xi_2$, $\eta=\eta_1+\eta_2$, where $\xi_i,\eta_i\in T_o\Sigma_i$. Note that $X$, $X_2\in\g{k}$, so $X_1\in\g{k}$ and thus $[X_1,\xi_2]\in\R i\xi_2$; since sections are totally real by Proposition~\ref{prop:realSection}, we get $\langle[X_1,\xi_2],\eta_2\rangle=0$. Then 
$\langle[X,\xi],\eta\rangle=\langle[X_1,\xi_1],\eta_1\rangle
+\langle[X_2,\xi_2],\eta_2\rangle$.

First, $X_1\in\pi_1(\g{h})$, so there exists $Y_2\in\g{u}(n-k)$ such that $X_1+Y_2\in\g{h}$. Hence, using Proposition~\ref{prop:criterion2} and the fact that $\xi_1\in\g{p}$, we get
\[
\langle[X_1,\xi_1],\eta_1\rangle=\langle[X_1+Y_2,\xi_1],\eta_1\rangle
=-\langle X_1+Y_2,[\xi_1,\eta_1]\rangle=0.
\]

Similarly, $X_2\in\pi_2(\g{h})$, so there exists $Y_1\in\g{u}(1,k)$ such that $Y_1+X_2\in\g{h}$. Again, by Proposition~\ref{prop:criterion2} and the fact that $\xi_2\in\g{p}$ we get
\[
\langle[X_2,\xi_2],\eta_2\rangle=\langle[Y_1+X_2,\xi_2],\eta_2\rangle
=-\langle Y_1+X_2,[\xi_2,\eta_2]\rangle=0.
\]

Altogether this implies $[\tilde{\g{h}}_o,\xi]\subset\nu_o(H\cdot o)\ominus T_o\Sigma$, and since $\g{h}_o\subset\tilde{\g{h}}_o$ the equality follows.

Now we show $[\hat{\g{h}}_o,\xi]=[\g{h}_o,\xi]$. We can use the same argument as above to prove that $[\hat{\g{h}}_o,\xi]\subset\nu_o(H\cdot o)\ominus T_o\Sigma$. The equality will follow after proving that $\g{h}_o\subset\hat{\g{h}}_o$. Hence take $X\in\g{h}_o=\g{h}\cap\g{k}$. Since $\g{h}\subset\g{u}(1,k)\oplus\g{u}(n-k)$, there exist $X_1\in\g{u}(1,k)$ and $X_2\in\g{u}(n-k)$ such that $X=X_1+X_2$. Since $\g{u}(n-k)\subset\g{k}$ this implies that $X_1\in\g{k}$. Thus, $X_2\in\pi_2(\g{h}_o)$ and $X_1\in\pi_1(\g{h}_o)\subset\pi_1(\g{h})$ as we wanted to prove. This completes the proof of the claim.

\medskip

Now note that the $\hat{H}$-action is proper. Indeed, since $U(n-k)$ is compact, $\pi_1(H)$ is closed. As $H_o$ is compact, $\pi_2(H_o)$ is also closed. Therefore, $\hat{H}$ is a closed subgroup of the isometry group of $\C H^n$, and hence its action on $\C H^n$ is proper.

The claim above implies that the tangent spaces at $\xi$ of the orbits of the slice representations of $H$ and $\tilde{H}$ at $o$ coincide for any $H$-regular vector $\xi\in T_o\Sigma$. Since $H\subset\tilde{H}$ we conclude that the principal orbits of the slice representation of $H$ at $o$ are orbits of the slice representation of $\tilde{H}$ at $o$. Similarly, the claim also shows that the principal orbits of the slice representation of $\hat{H}$ at $o$ are orbits of the slice representation of $\tilde{H}$ at~$o$. Now, if we replace $\tilde{H}=\pi_1(H)\times \pi_2(H)$ by its closure in $U(1,k)\times U(n-k)$, the closed orbits of its slice representation at $o$ remain unchanged, in particular, those which agree with the principal orbits of the slice representations at $o$ of $H$ and $\hat{H}$. Thus, we can apply Lemma~\ref{lemma:sameslice} separately to $H$ and the closure of $\tilde{H}$, and to $\hat{H}$ and the closure of $\tilde{H}$. We conclude in particular that the actions on $\C H^n$ of the closed subgroups $H$ and $\hat{H}$ of $U(1,n)$ have the same orbits. Altogether this proves (iii).

\medskip

Now let us prove the other direction of the equivalence in Theorem~\ref{theorem:ProdAct}. Assume $H
\subset L$ is a closed subgroup such that (i), (ii) and (iii) hold.
Because of~(iii) we may replace $H$ by $\pi_1(H) \times \pi_2(H_o)$.
Let $\Sigma_1$ be the section of the $H$-action on~$M_1$ and let
$\Sigma_2$ be the section of the $H_o$-action on~$M_2$. Since both actions are nontrivial, by
Proposition~\ref{prop:realSection} the tangent spaces $T_o\Sigma_1$
and $T_o\Sigma_2$ are totally real subspaces of $T_o\C H^n$;
moreover, $\C T_o\Sigma_1 \perp \C T_o\Sigma_2$. Thus the sum
$T_o\Sigma_1 \oplus T_o\Sigma_2$ is a totally real Lie triple system in
$T_o\C H^n$. Let $\Sigma$ be the corresponding totally geodesic
submanifold.

Using Proposition~\ref{prop:criterion1}, we will show that the $H$-action
on~$\C H^n$ is polar and $\Sigma$ is a section. Consider the Cartan
decomposition $\g{g} = \g{k} \oplus \g{p}$ with respect to $o\in \C H^n$. We
have $\g{p} = T_o M_1 \oplus T_o M_2$. Furthermore, the direct sum
decomposition
\begin{equation}\label{EqnSplitNSpace}
\nu_o (H \cdot o) = (\nu_o (H \cdot o) \cap T_oM_1) \oplus T_oM_2
\end{equation}
holds. The slice representation of the $H$-action on~$M_1$ at the point~$o$ is
orbit equivalent to the submodule $\nu_o (H \cdot o) \cap T_oM_1$ of the slice
representation of the $H$-action on~$\C H^n$ at~$o$. The slice representation
of the $H_o$-action on~$M_2$ at the point~$o$ is orbit equivalent to the
submodule $T_oM_2$ of the slice representation of the $H$-action on~$\C H^n$
at~$o$. By \cite[Theorem~4]{Da85}, we conclude that the slice representation of
$H_o$ on~$\nu_o(H \cdot o)$ is polar and a section is $T_o \Sigma = T_o\Sigma_1
\oplus T_o\Sigma_2$. We have to show $\langle [v,w], X \rangle =
-B([v,w],\theta(X)) = 0$ for all $v$, $w \in T_o\Sigma \subset \g{p}$ and all
$X \in \g{h}$. We may identify the tangent space $T_o\C H^n = \g{p}$ with the
space of complex $(n+1) \times (n+1)$-matrices of the form
\begin{equation}\label{EqnTangentSpace}
\left(
  \begin{array}{c|ccc}
    0 & \bar z_1 & \dots & \bar z_n \\ \hline
    z_1 & 0 & \dots & 0 \\
    \vdots & \vdots & & \vdots \\
    z_n & 0 & \dots & 0 \\
  \end{array}
\right).
\end{equation}
The subspace $T_oM_1$ is given by the matrices where $z_{k+1} =
\ldots = z_n = 0$. On the other hand, $T_oM_2$ consists of those
matrices where $z_1 = \ldots = z_k = 0$. Let $v,w \in T_o\Sigma_1$.
Then $[v,w]$ is a matrix all of whose nonzero entries are located in
the $(k+1) \times (k+1)$-submatrix in the upper left-hand corner, and
it follows from~(i) and Proposition~\ref{prop:criterion1} that all
vectors in $\g{h}$ are orthogonal to $[v,w]$. Now assume $v,w \in
T_o\Sigma_2$. Then $[v,w]$ is a matrix all of whose nonzero entries
are located in the $(n-k) \times (n-k)$-submatrix in the bottom
right-hand corner. It follows from~(ii) and
Proposition~\ref{prop:criterion1} that all vectors in $\g{h}$ are
orthogonal to $[v,w]$. Finally assume $v \in T_o\Sigma_1$ and $w \in
T_o\Sigma_2$. In this case, the bracket $[v,w]$ is contained in the
orthogonal complement of the Lie algebra of~$L$ in $\g{su}(1,n)$; in
particular, $[v,w]$ is orthogonal to~$\g{h}$. We conclude that the
$H$-action on $\C H^n$ is polar by Proposition~\ref{prop:criterion1}.
\end{proof}

\subsection{Actions leaving a totally geodesic real hyperbolic space invariant}
\label{subsec:RHnInvariant}

Now we assume that the polar action leaves a totally geodesic $\R
H^n$ invariant. We have:

\begin{theorem}\label{th:SubSO}
Assume that $H$ is a closed subgroup of $SO^0(1,n) \subset SU(1,n)$. If
the $H$-action on~$\C H^n$ is polar and nontrivial, then it is orbit
equivalent to the $SO^0(1,n)$-action on $\C H^n$; in particular, it is
of cohomogeneity one.
\end{theorem}

\begin{proof}
This proof is divided in three steps.

\begin{claim}\label{claim:SubSO}
The group $H$ induces a homogeneous polar foliation on the totally
geodesic submanifold $\R H^n$ given by the $SO^0(1,n)$-orbit
through~$o$.
\end{claim}

Let $M_1$ be the totally geodesic $\R H^n$ given by the
$SO^0(1,n)$-orbit through $o$. Obviously, the $H$-action leaves $M_1$
invariant. Assume the $H$-action on~$M_1$ has a singular orbit $H
\cdot p$, where $p = g(o) \in M_1$. Consider the action of $H'$ on
$\C H^n$, where $H'$ is the conjugate subgroup $H' = gHg^{-1}$
of~$SU(1,n)$. The action of $H'$ is conjugate to the $H$-action on
$\C H^n$, hence polar. We have the splitting~(\ref{EqnSplitNSpace})
for the normal space of the $H'$-orbit through $o$ as in the proof of
Theorem~\ref{theorem:ProdAct}, where in this case $M_2$ is the totally
geodesic $\R H^n$ such that $T_oM_2 = i(T_oM_1)$. Since $o$ is contained in a singular orbit of the $H'$-action on~$M_1$, the slice representation
of $H'_o$ on $V = \nu_o(H' \cdot o) \cap T_oM_1$ is nontrivial. The
space $T_oM_1$ consists of all matrices in~(\ref{EqnTangentSpace})
where the entries $z_1, \ldots, z_n$ are real. Consequently, the
space $iV$ is contained in the normal space $\nu_o(H' \cdot o)$ and
it follows that the slice representation of $H'_o$ with respect to
the $H'$-action on~$\C H^n$ contains the submodule $V \oplus iV$
with two equivalent nontrivial $H'_o$-representations and is hence
non-polar by \cite[Lemma~2.9]{k02}, a contradiction. Hence the
$H$-action on $M_1$ does not have singular orbits, i.e.\ $H$ induces
a homogeneous foliation on~$M_1$.

\begin{claim}\label{claim:OneLeaf}
The homogeneous polar foliation induced on the invariant totally geodesic real
hyperbolic space consists of only one leaf or all the leaves are points.
\end{claim}

Consider the point $o \in M_1$ as in the proof of
Claim~\ref{claim:SubSO}. The tangent space of~$M_1$ at~$o$ splits as
\[
T_oM_1 = T_o(H \cdot o) \oplus (\nu_o(H \cdot o) \cap T_oM_1).
\]
The action of the isotropy group $H_o$ on $T_oM_1$ respects this
splitting. Moreover, the action is trivial on $V = \nu_o(H \cdot o)
\cap T_oM_1$, as this is a submodule of the slice representation
at~$o$, which lies in a principal orbit of the $H$-action on $M_1$.
It follows that the action of $H_o$ on $iV$  is trivial as well and
the only possibly nontrivial submodule of the slice representation
at~$o$ is $iW$, where we define $W = T_o(H \cdot o)$. It follows
that the action of the isotropy group $H_o$ on $iW$ is polar by
Proposition~\ref{prop:criterion1}. Let $\Sigma'$ be a section of this
action. Let $\Sigma$ be a section of the $H$-action on $\C H^n$. Then
we have
\[
T_o\Sigma = V \oplus iV \oplus \Sigma'.
\]
By Proposition~\ref{prop:realSection}, $\Sigma$ is either totally
real or $\Sigma=\C H^n$.  In the first case, $V$ must be $0$, so the
action of $H$ on $M_1$ is transitive. In the second case, the action
of $H$ on $\C H^n$ is trivial.

\begin{claim}
The $H$-action on $\C H^n$ is orbit equivalent to the
$SO^0(1,n)$-action.
\end{claim}

Assume the $H$-action is nontrivial and polar with section~$\Sigma$.
We will use the notation of Subsection~\ref{subsec:CHn}. By
Claim~\ref{claim:OneLeaf}, $H$ acts transitively on $M_1 = \R H^n$.
By Lemma~\ref{lemma:equivariance}, the tangent space $T_o(H\cdot
o)=T_o M_1$ coincides with $\g{a}\oplus(1-\theta)\g{g}_\alpha^\R$,
where $\g{g}_\alpha^\R$ is a totally real subspace of the root space
$\g{g}_\alpha$ satisfying $\C \g{g}_\alpha^\R = \g{g}_\alpha$.
Moreover $\nu_o M_1 = i(T_o M_1)$. The action of the isotropy
subgroup $H_o = H \cap K$ on $\nu_o M_1$ by the slice representation
is polar with section $T_o\Sigma$. Since $iB \in \nu_oM_1$, by
conjugating the section with a suitable element in $H_o$ we can then
assume that $iB\in T_o\Sigma$.

According to~\cite[Proposition~2.2]{BDT10}, the group~$H$ contains a
solvable subgroup~$S$ which acts transitively on $M_1 = \R H^n$.
Since $S$ is solvable, it is contained in a Borel subgroup of
$SO^0(1,n)$. As shown in the proof of~\cite[Proposition~4.2]{BD11b}, we
may assume that the Lie algebra of such a Borel subgroup is maximally
noncompact, i.e.\ its Lie algebra is $\g{t} \oplus\g{a} \oplus
\g{g}_\alpha^\R$, where $\g{t}$ is an abelian subalgebra of~$\g{k}
\cap \g{so}(n)$ such that $\g{t}\oplus\g{a}$ is a Cartan subalgebra
of $\g{so}(1,n)$, see~\cite{mostow}. Note that the Cartan
decomposition of $\g{so}(1,n)$ with respect to the point $o\in M_1=\R
H^n$ is $\g{so}(1,n)=(\g{k}\cap\g{so}(1,n))\oplus\g{p}^\R$, where
$\g{p}^\R=\g{a}\oplus (1-\theta)\g{g}_\alpha^\R\cong T_o M_1$, and
$\g{g}_\alpha^\R$ is the only positive root space of $\g{so}(1,n)$
with respect to the maximal abelian subalgebra $\g{a}$ of $\g{p}^\R$,
for a fixed order in the roots.

Now assume the $H$-action on~$\C H^n$ is not of cohomogeneity one.
Then $T_o \Sigma \subset \nu_o M_1$ is a Lie triple system containing
$iB$ and a nonzero vector $iw$ such that $iB$, $iw \in \g{p}$ are
orthogonal. By Lemma~\ref{lemma:equivariance}, there is a vector
$W\in\g{g}_\alpha^\R$ such that $w=(1-\theta)W$. Then, using
Lemma~\ref{lemma:aux}\eqref{lemma:aux:a}, we have
\[
[iB,iw] = \frac{1}{2}[(1-\theta)Z,(1-\theta)JW]
=\frac{1}{2}(1+\theta)[\theta JW,Z]=\frac{1}{2}(1+\theta)W.
\]
Since $T_o(S\cdot o)=T_oM_1$, it follows that the orthogonal
projection  of the Lie algebra of $S$ onto $\g{p}$ is
$\g{p}^\R=\g{a}\oplus(1-\theta)\g{g}_\alpha^\R$. This implies that
$\g{a}\oplus\g{g}_\alpha^\R$ is contained in the Lie algebra of $S$,
and hence, also in $\g{h}$. But then $W\in \g{h}$ and
\[
\langle [iB,iw], W\rangle=\frac{1}{2}\langle (1+\theta)W, W\rangle
=\frac{1}{2}\langle W, W\rangle\neq 0,
\]
so we have arrived at a contradiction with the criterion for polarity
in Proposition~\ref{prop:criterion1}.
\end{proof}


\section{The parabolic case}\label{sec:parabolic}

As above, let $G=SU(1,n)$ be the identity connected component of the isometry
group of $\C H^n$, and $K=S(U(1)U(n))$ the isotropy group at some point $o$.
Let $\g{g}=\g{k}\oplus\g{p}$ be the Cartan decomposition of the Lie algebra of
$G$ with respect to $o$, and choose a maximal abelian subspace $\g{a}$ of
$\g{p}$. As usual we consider $\g{n}=\g{g}_\alpha\oplus\g{g}_{2\alpha}$, where
$\alpha$ is a simple positive restricted root. Recall from Subsection~\ref{subsec:CHn} that $\g{k}_0$ normalizes both $\g{a}$ and $\g{n}$. Then $\g{k}_0\oplus\g{a}\oplus\g{n}$ is a
maximal parabolic subalgebra, and a maximal parabolic subgroup can be written
as the semi-direct product $K_0 A N$.
%

The aim of this section is to prove the following decomposition theorem.

\begin{theorem}\label{th:Parabolic}
Let $H$ be a connected closed subgroup of $K_0 A N$ acting polarly
and nontrivially on $\C H^n$. Then the action of $H$ is orbit
equivalent to the action of a subgroup of $K_0 A N$ whose Lie algebra
can be written as one of the following:
\begin{enumerate}[{\rm (a)}]
\item $\g{q}\oplus\g{a}$, where $\g{q}$ is a subalgebra
    of $\g{k}_0$.\label{th:Parabolic:qaw}

\item $\g{q}\oplus\g{a}\oplus\g{w}\oplus\g{g}_{2\alpha}$, where
    $\g{w}$ is a subspace of $\g{g}_\alpha$, and
    $\g{q}$ is a subalgebra of
    $\g{k}_0$.\label{th:Parabolic:qawg}

\item $\g{q}\oplus\g{w}\oplus\g{g}_{2\alpha}$, where $\g{w}$ is
    a subspace of $\g{g}_\alpha$, and $\g{q}$ is a
    subalgebra of $\g{k}_0$.\label{th:Parabolic:qwg}
\end{enumerate}
\end{theorem}

Recall that an isometric action on a complete connected Riemannian
manifold by a closed subgroup of its isometry group is a proper
action. In particular, this implies that isotropy groups are compact,
orbits are closed, and the orbit space is Hausdorff. One can then
talk about types of orbits: two orbits have the same type if the
isotropy groups at any given points of these orbits are conjugate.
The set of conjugacy classes of isotropy groups of orbits is a
partially ordered set by inclusion. For a proper action there is
always a maximum orbit type, that is, a type of orbit whose isotropy
groups are contained, up to conjugacy, in the isotropy groups of any
of the other orbits. Orbits belonging to this type are called
principal orbits. They have the maximum possible dimension, and their
union constitutes an open dense subset of the ambient manifold.
Orbits that are not principal but have maximum dimension are called
exceptional. The rest of the orbits are said to be singular. For
proper isometric actions on Hadamard manifolds there is also a
minimum orbit type.

\begin{proposition}\label{prop:MinimumOrbitType}
Let $M$ be a Hadamard manifold and $H$ a closed subgroup of its
isometry group acting on $M$. Then, there is a minimum orbit type,
that is, there is an orbit type whose isotropy groups contain, up to
conjugation in $H$, the isotropy groups of any
other orbits.
\end{proposition}

\begin{proof}
Let $Q$ be a maximal compact subgroup of $H$. Any two maximal compact
subgroups of a connected Lie group $H$ are connected and conjugate by
an element of $H$ \cite[p.~148--149]{OV94}. By Cartan's fixed point
theorem, $Q$ fixes a point $p\in M$, and hence $Q=H_p$, the isotropy
group of $H$ at $p$. If  $q\in M$, then $H_q$ is compact, and since
all maximal compact subgroups of $H$ are conjugate it follows that
there exists $h\in H$ such that $H_q\subset hQh^{-1}$. Thus, the
orbit through $p$ is of the minimum orbit type.
\end{proof}

Coming back to the problem in $\C H^n$, consider from now on a
connected closed subgroup $H$ of $K_0AN$ acting polarly and
nontrivially on $\C H^n$. Proposition~\ref{prop:MinimumOrbitType} and
its proof assert that there is a maximal compact subgroup $Q$ of $H$
with a fixed point $p\in\C H^n$. The orbit through $p$ is of minimum
type, and $Q=H_p$. Since $AN$ acts simply transitively on $\C H^n$,
we can take the unique element $g$ in $AN$ such that $g(o)=p$, and
consider the group $H'=I_{g^{-1}}(H)=g^{-1}Hg$, whose action on $\C
H^n$ is conjugate to the one of $H$. Moreover,
$Q'=I_{g^{-1}}(Q)=g^{-1}Qg$ fixes the point $o$. Since
$\g{a}\oplus\g{n}$ normalizes $\g{k}_0\oplus\g{a}\oplus\g{n}$, we get
that $AN$ normalizes $\g{k}_0\oplus\g{a}\oplus\g{n}$. In particular,
$\Ad(g^{-1})\g{h}\subset \g{k}_0\oplus\g{a}\oplus\g{n}$, and
therefore $H'\subset K_0 A N$. Since we are interested in the study
of polar actions up to orbit equivalence, it is not restrictive to
assume that the group $H\subset K_0 A N$ acting polarly on $\C H^n$
admits a maximal connected compact subgroup $Q$ that fixes the point
$o$, and hence $Q\subset K_0$. We will assume this throughout this
section.

As a matter of notation, given two subspaces $\g{m}$, $\g{l}$, and a
vector $v$ of $\g{g}$, by $\g{m}_\g{l}$ (resp.\ by $v_\g{l}$) we will
denote the orthogonal projection of $\g{m}$ (resp.\ of $v$) onto
$\g{l}$.

The crucial part of the proof of Theorem~\ref{th:Parabolic} is contained
in the following assertion:

\begin{proposition}\label{prop:orbit equivalent}
Let $H$ be a connected closed subgroup of $K_0 A N$ acting polarly on $\C H^n$.
Let $Q$ be a maximal subgroup of $H$ that fixes the point $o\in\C H^n$. Let
$\g{b}$ be a subspace of $\g{a}$, $\g{w}$ a subspace of $\g{g}_\alpha$, and
$\g{r}$ a subspace of $\g{g}_{2\alpha}$. Assume that
$\hat{\g{h}}=\g{q}\oplus\g{b}\oplus\g{w}\oplus\g{r}$ is a subalgebra of
$\g{k}_0\oplus\g{a}\oplus\g{n}$, and let $\hat{H}$ be the connected subgroup of
$K_0 A N$ whose Lie algebra is $\hat{\g{h}}$. If
$\g{h}_{\g{a}\oplus\g{n}}=\g{b}\oplus\g{w}\oplus\g{r}$, then the actions of $H$
and $\hat{H}$ are orbit equivalent.
\end{proposition}

The proof of Proposition~\ref{prop:orbit equivalent} is carried out
in several steps. We start with a few basic remarks.

Since $\g{a}$ and $\g{g}_{2\alpha}$ are one dimensional, $\g{b}$ is
either $0$ or $\g{a}$, and $\g{r}$ is either $0$ or
$\g{g}_{2\alpha}$. Moreover, if $\g{r}=0$ then $\g{w}$ has to be a
totally real subspace of the complex vector space
$\g{g}_\alpha\cong\C^{n-1}$, so that $\hat{\g{h}}$ is a Lie
subalgebra. Using the properties of the root space decomposition, it
is then easy to check that
$\hat{\g{h}}=\g{q}\oplus\g{b}\oplus\g{w}\oplus\g{r}$ is a subalgebra
of $\g{k}_0\oplus\g{a}\oplus\g{n}$ if and only if
$[\g{q},\g{w}]\subset\g{w}$.

Let $\Sigma$ be a section of the action of $H$ on $\C H^n$ through
$o\in\C H^n$, and let $T_o \Sigma$ be its tangent space at $o$. The
normal space of the orbit through the origin is $\nu_o(H\cdot
o)=(\g{a}\ominus\g{b})\oplus(\g{p}_\alpha\ominus(1-\theta)\g{w})
\oplus(\g{p}_{2\alpha}\ominus(1-\theta)\g{r})$. Since
$[\g{k}_0,\g{a}]=[\g{k}_0,\g{g}_{2\alpha}]=0$,
$[\g{k}_0,\g{g}_\alpha]=\g{g}_\alpha$, and $\nu_o(H\cdot
o)=T_o\Sigma\oplus[\g{q},T_o\Sigma]$ (orthogonal direct sum of vector
subspaces) by Proposition~\ref{prop:criterion2}, it follows that
$\g{a}\ominus\g{b}\subset T_o\Sigma$ and
$\g{p}_{2\alpha}\ominus(1-\theta)\g{r}\subset T_o\Sigma$. Moreover,
since sections are totally real by
Proposition~\ref{prop:realSection}, we can write the tangent space at
$o$ of any section as
$T_o\Sigma=(\g{a}\ominus\g{b})\oplus(1-\theta)\g{s}
\oplus(\g{p}_{2\alpha}\ominus(1-\theta)\g{r})$, where $\g{s}$ is a
totally real subspace of $\g{g}_\alpha$, with
$\g{s}\subset\g{g}_\alpha\ominus\g{w}$. Furthermore, the fact that
$T_o\Sigma$ is totally real, and $i\g{a}=\g{p}_{2\alpha}$ (where $i$
is the complex structure on $\g{p}$), implies that
$\g{a}\ominus\g{b}=0$ or $\g{p}_{2\alpha}\ominus(1-\theta)\g{r}=0$,
or equivalently, $\g{b}=\g{a}$ or $\g{r}=\g{g}_{2\alpha}$ (that is,
$\g{a}\subset\g{h}_{\g{a}\oplus\g{n}}$ or
$\g{g}_{2\alpha}\subset\g{h}_{\g{a}\oplus\g{n}}$).

Let $T+aB+U+xZ$ be an arbitrary element of $\g{h}$, with
$T\in\g{h}_{\g{k}_0}$, $U\in\g{w}$, and $a,x\in\R$. Let $\xi$, $\eta$
be arbitrary vectors of $\g{s}$. By
Proposition~\ref{prop:criterion2}, and since $\g{s}$ is totally real,
we have, using Lemma~\ref{lemma:aux}(\ref{lemma:aux:b}):
\[
0=\langle T+aB+U+xZ,[(1-\theta)\xi,(1-\theta)\eta]\rangle
=-\langle T,(1+\theta)[\theta \xi,\eta]\rangle=-2\langle [T,\xi],\eta\rangle,
\]
from where it follows that $[\g{h}_{\g{k}_0},\g{s}]\subset \g{g}_\alpha\ominus\g{s}$.

Moreover, if $T\in\g{q}$ and $S_U\in\g{h}_{\g{k}_0}$, $U\in\g{w}$ are
such that $S_U+U\in\g{h}$, then $[T,S_U]+[T,U]=[T,S_U+U]\in\g{h}$, so
$[T,U]\in\g{w}$. In particular, if $\xi\in\g{s}$, then $0=\langle
[T,U],\xi\rangle=-\langle [T,\xi],U\rangle$, which proves
$[\g{q},\g{s}]\subset\g{g}_\alpha\ominus(\g{w}\oplus\g{s})$.

Summarizing what we have obtained about sections we can state:

\begin{lemma}\label{lemma:section1}
If $\Sigma$ is a section of the action of $H$ on $\C H^n$
through $o$, then
\[
T_o\Sigma=(\g{a}\ominus\g{b})\oplus(1-\theta)\g{s}
\oplus(\g{p}_{2\alpha}\ominus(1-\theta)\g{r}),
\]
where $\g{s}\subset\g{g}_\alpha\ominus\g{w}$ is a totally real
subspace of $\g{g}_\alpha$, and $\g{b}=\g{a}$ or
$\g{r}=\g{g}_{2\alpha}$. Moreover, $[\g{h}_{\g{k}_0},\g{s}]\subset
\g{g}_\alpha\ominus\g{s}$, and
$[\g{q},\g{s}]\subset\g{g}_\alpha\ominus(\g{w}\oplus\g{s})$.
\end{lemma}

We will need to calculate the isotropy group at certain points.

\begin{lemma}\label{lemma:isotropy}
Let $\xi\in\g{g}_\alpha$ and write $g=\Exp(\lambda\xi)$, with
$\lambda\in\R$. Then, the Lie algebra of the isotropy group $H_p$ of $H$
at $p=g(o)$ is $\g{h}_p=\g{h}\cap\Ad(g)\g{k}=\g{q}\cap\ker\ad(\xi)$.
\end{lemma}
\begin{proof}
First notice that $\g{h}\cap\Ad(g)\g{k}$ is the Lie algebra of
$H_p=H\cap I_g(K)$. Let $v$ be the unique element in $\g{p}=T_o\C
H^n$ such that $\exp_{o}(v)=p$. We show that the isotropy group
$H_p$ coincides with the isotropy group  of the slice representation
of $Q$ at $v$, $Q_{v}$. By \cite[\S2]{To07} we know that the
normal exponential map $\exp\colon\nu(H\cdot o)\to\C H^n$ is an
$H$-equivariant diffeomorphism. Let $h\in H_p$. Since
$\exp_o(v)=p=h(p)=h\exp_o(v)=\exp_{h(o)}(h_{*o}v)$, we get
that $h(o)=o$ and $h_{*o}v=v$, and hence, $h\in Q_{v}$. The
$H$-equivariance of $\exp$ also shows the converse inclusion.
Therefore $H_p=Q_{v}$.

We can write $v=aB+b(1-\theta)\xi$ for certain $a$, $b\in\R$. In
fact, $\Exp(\lambda\xi)(o)$ belongs to the totally geodesic $\R H^2$
given by $\exp_o(\g{a}\oplus\R (1-\theta)\xi)$, and $b\neq 0$ if
$\lambda\xi\neq 0$. Then, the Lie algebra of $H_p=Q_{v}$ is
$\{T\in\g{q}:[T,aB+b(1-\theta)\xi]=0\}=\{T\in\g{q}:[T,\xi]=0\}$,
which is $\g{q}\cap\ker \ad(\xi)$.
\end{proof}

By definition, we say that a vector $\xi\in\g{s}$ is \emph{regular} if
$[\g{q},\xi]=\g{g}_\alpha\ominus(\g{w}\oplus\g{s})$. We have

\begin{lemma}
The set $\{\xi\in\g{s}:\text{$\xi$ is regular}\}$ is an open dense
subset of $\g{s}$.
\end{lemma}

\begin{proof}
An element of $T_o\Sigma$ can be written, according to
Lemma~\ref{lemma:section1}, as $v=aB+(1-\theta)\xi+x(1-\theta)Z$
where $a$, $x\in\R$, and $\xi\in\g{s}$. We have
$[\g{q},v]=(1-\theta)[\g{q},\xi]$ and $\nu_o(H\cdot o)\ominus
T_o\Sigma=(1-\theta)(\g{g}_\alpha\ominus(\g{w}\oplus\g{s}))$. An
element of $T_o\Sigma$ is regular (that is, belongs to a principal
orbit of the slice representation $Q\times\nu_o(H\cdot
o)\to\nu_o(H\cdot o)$) if and only if $[\g{q},v]=\nu_o(H\cdot
o)\ominus T_o\Sigma$. The previous equalities, and the fact that
$(1-\theta)\colon\g{g}_\alpha\to\g{p}_\alpha$ is an isomorphism
implies that $v$ is regular if and only if
$[\g{q},\xi]=\g{g}_\alpha\ominus(\g{w}\oplus\g{s})$. Since the set of
regular points of a section is open and dense, the result follows.
\end{proof}


\begin{lemma}\label{lemma:normalizes}
For each regular vector $\xi\in\g{s}$ we have
$[\g{h}_{\g{k}_0},\xi]=\g{g}_\alpha\ominus(\g{w}\oplus\g{s})$.
\end{lemma}

\begin{proof}
Let $\xi\in\g{s}$ be a regular vector, that is,
$[\g{q},\xi]=\g{g}_\alpha\ominus(\g{w}\oplus\g{s})$. In order to
prove the lemma, it is enough to show that
$[\g{h}_{\g{k}_0},\xi]\subset\g{g}_\alpha\ominus\g{w}$, since
$\g{q}\subset\g{h}_{\g{k}_0}$ and, by Lemma~\ref{lemma:section1},
$[\g{h}_{\g{k}_0},\xi]\subset \g{g}_\alpha\ominus\g{s}$.

First, consider the case $\g{r}=0$. By Lemma~\ref{lemma:section1},
$T_o\Sigma=(1-\theta)\g{s}\oplus\R(1-\theta)Z$ for each section
$\Sigma$ through $o$, where $\g{s}$ is some totally real subspace of
$\g{g}_\alpha$. By Proposition~\ref{prop:criterion2} we have
$\nu_o(H\cdot o)=\Ad(Q)(T_o\Sigma)$ and, thus, for any
$\eta\in\g{g}_\alpha\ominus\g{w}$ we can find a section $\Sigma$
through $o$ such that $\eta\in \g{s}$ by conjugating by a suitable
element in $Q$. Then using Lemma~\ref{lemma:aux}, we have that
$(1+\theta)J\eta=[(1-\theta)\eta,(1-\theta)Z]\in
[T_o\Sigma,T_o\Sigma]$. Let $W\in\g{w}$ and $T_W\in\g{h}_{\g{k}_0}$
be such that $T_W+W\in\g{h}$. Since by
Proposition~\ref{prop:criterion2} we have
$\langle\g{h},[T_o\Sigma,T_o\Sigma]\rangle=0$, then $0=\langle
T_W+W,(1+\theta)J\eta\rangle=\langle W, J\eta\rangle$. We have then
shown that $J(\g{g}_\alpha\ominus\g{w})$ is orthogonal to $\g{w}$,
that is, $\g{g}_\alpha\ominus\g{w}$ is a complex subspace of
$\g{g}_\alpha$. Since $\g{w}$ is totally real, we deduce $\g{w}=0$.
But then $[\g{h}_{\g{k}_0},\xi]\subset\g{g}_\alpha\ominus\g{w}$ holds
trivially.

For the rest of the proof, we assume that $\g{r}=\g{g}_{2\alpha}$.

Let $T_B\in\g{h}_{\g{k}_0}$ and $a\in\R$ such that $T_B+aB\in\g{h}$.
Note that, if $\g{b}=0$, then $a=0$, $T_B\in\g{q}$ and there is
nothing to prove. For each $U\in \g{w}$ take an
$S_U\in\g{h}_{\g{k}_0}$ with $S_U+U\in\g{h}$. Then
$[T_B,S_U]+[T_B,U]+\frac{a}{2}U=[T_B+aB, S_U+U]\in \g{h}$, so
$[T_B,U]+\frac{a}{2}U\in\g{w}$, from where $[T_B,U]\in\g{w}$. Hence,
$\langle [T_B,\xi],U\rangle=-\langle\xi,[T_B,U]\rangle=0$, so we get
$[T_B,\xi]\in\g{g}_\alpha\ominus\g{w}$.

Now let $T_Z\in\g{h}_{\g{k}_0}$ and $x\in Z$ with $T_Z+xZ\in\g{h}$.
For each $U\in \g{w}$ take an $S_U\in\g{h}_{\g{k}_0}$ with
$S_U+U\in\g{h}$. Then $[T_Z,S_U]+[T_Z,U]=[T_Z+Z, S_U+U]\in \g{h}$, so
$[T_Z,U]\in\g{w}$. As above, we conclude
$[T_Z,\xi]\in\g{g}_\alpha\ominus\g{w}$.

Finally, we have to prove that for each $U\in\g{w}$, if
$T_U\in\g{h}_{\g{k}_0}$ is such that $T_U+U\in\g{h}$, then
$[T_U,\xi]\in\g{g}_\alpha\ominus\g{w}$. This will require some
effort.

Let $U\in\g{w}$ and $T_U'\in\g{h}_{\g{k}_0}$ with $T_U'+U\in\g{h}$. By
Lemma~\ref{lemma:section1}, $[T_U',\xi]\in\g{g}_\alpha\ominus\g{s}
=\g{w}\oplus(\g{g}_\alpha\ominus(\g{w}\oplus\g{s}))$. Since
$[\g{q},\xi]=\g{g}_\alpha\ominus(\g{w}\oplus\g{s})$, we can find an
$S\in\g{q}$ so that $[T_U'+S,\xi]\in\g{w}$. Thus, if we define $T_U=T_U'+S$, we still have $T_U\in\g{h}_{\g{k}_0}$ and $T_U+U\in\g{h}$, but also $[T_U,\xi]\in\g{w}$.
Therefore we can define
the endomorphism of~$\g{w}$
\[
F_\xi\colon\g{w}\to\g{w},\quad U\mapsto [T_U,\xi],\quad
\text{where $T_U\in\g{h}_{\g{k}_0}$, $T_U+U\in\g{h}$, and $[T_U,\xi]\in\g{w}$.}
\]
The map $F_\xi$ is well-defined. Indeed, if $T_U$,
$S_U\in\g{h}_{\g{k}_0}$, $U\in\g{w}$, $T_U+U$, $S_U+U\in\g{h}$, and
$[T_U,\xi]$, $[S_U,\xi]\in\g{w}$, then $T_U-S_U\in\g{q}$, so
$[T_U,\xi]-[S_U,\xi]=[T_U-S_U,\xi]\in\g{g}_\alpha\ominus(\g{w}\oplus\g{s})$,
and $[T_U,\xi]-[S_U,\xi]\in\g{w}$. Hence $[T_U,\xi]=[S_U,\xi]$. It is
also easy to check that $F_\xi$ is linear.

Furthermore, $F_\xi$ is self-adjoint. To see this, let $T_U$,
$S_V\in\g{h}_{\g{k}_0}$, $U$, $V\in\g{w}$, with $T_U+U$,
$S_V+V\in\g{h}$, and $[T_U,\xi]$, $[S_V,\xi]\in\g{w}$. Then we have
\begin{align*}
0&=\langle [T_U+U,S_V+V],\xi\rangle
=\langle[T_U,V],\xi\rangle-\langle[S_V,U],\xi\rangle
=-\langle V, [T_U,\xi]\rangle+\langle U,[S_V,\xi]\rangle\\
&=-\langle F_\xi(U),V\rangle+\langle F_\xi(V), U\rangle.
\end{align*}

Assume now that $F_\xi\neq 0$. Then $F_\xi$ admits an
eigenvector $U\in\g{w}$ with nonzero eigenvalue $\lambda\in\R$:
$F_\xi(U)=\lambda U\neq 0$. We will get a contradiction with this.

Let $g=\Exp(-\frac{1}{\lambda}\xi)$, and consider
$T_U\in\g{h}_{\g{k}_0}$ such that $T_U+U\in\g{h}$ and
$F_\xi(U)=[T_U,\xi]=\lambda U$. We also consider an element
$S\in\g{h}_{\g{k}_0}$ such that $S+Z\in\g{h}$ and $[S,\xi]=0$; this
is possible because
$[S,\xi]\in\g{g}_\alpha\ominus(\g{w}\oplus\g{s})=[\g{q},\xi]$ and
$\g{q}\subset\g{h}$. If we define $R=T_U-\frac{1}{4\lambda}\langle
J\xi,U\rangle S\in\g{h}_{\g{k}_0}$, then we have
\begin{align*}
\Ad(g)R&=e^{-\frac{1}{\lambda}\ad(\xi)}R
= T_U - \frac{1}{\lambda}[\xi,T_U]+\frac{1}{2\lambda^2}[\xi,[\xi,T_U]]
-\frac{1}{4\lambda}\langle J\xi,U\rangle S\\
&= (T_U + U) -\frac{1}{4\lambda}\langle J\xi, U\rangle (S+Z)\in\g{h}\cap\Ad(g)(\g{k}).
\end{align*}
However, $\Ad(g)R\not\in\g{q}\cap\ker\ad(\xi)$.  By virtue of
Lemma~\ref{lemma:isotropy}, this gives a contradiction. Thus we must
have $F_\xi=0$, from where the result follows.
\end{proof}

\begin{lemma}\label{lemma:normalizes2}
The subspace $\g{h}_{\g{k}_0}$ is a subalgebra of $\g{k}_0$ and
$[\g{h}_{\g{k}_0},\g{w}]\subset\g{w}$.
\end{lemma}

\begin{proof}
If $T+aB+U+xZ$, $S+bB+V+yZ\in\g{h}$, with $T$, $S\in\g{h}_{\g{k}_0}$,
$U$, $V\in\g{w}$, and $a$, $b$, $x$, $y\in\R$, then the bracket
$[T+aB+U+xZ,S+bB+V+yZ]=[T,S]+[T,V]-[S,U]+\frac{a}{2}V-\frac{b}{2}U
+\left(\frac{1}{2}\langle JU,V\rangle +ay-bx\right)Z$ belongs to
$\g{h}$. In particular $[T,S]\in\g{h}_{\g{k}_0}$, so
$\g{h}_{\g{k}_0}$ is a Lie subalgebra of $\g{k}_0$. Taking $U=0$,
$a=b=x=y=0$ we obtain that $[\g{q},\g{w}]\subset\g{w}$ and hence
$[\g{q},\g{g}_\alpha\ominus\g{w}]\subset\g{g}_\alpha\ominus\g{w}$.

Now let $X\in\g{g}_\alpha\ominus\g{w}$. For any section through $o$
we have $\Ad(Q)(T_o\Sigma)=\nu_0(H\cdot
o)=(\g{a}\ominus\g{b})\oplus(1-\theta)(\g{g}_\alpha\ominus\g{w})
\oplus(1-\theta)(\g{g}_{2\alpha}\ominus\g{r})$, and $(\g{a}\ominus\g{b})
\oplus(1-\theta)(\g{g}_\alpha\ominus\g{r})\subset T_o\Sigma$ by
Lemma~\ref{lemma:section1}. Hence, for
$(1-\theta)X\in(1-\theta)(\g{g}_\alpha\ominus\g{w})$ we can find a
section $\Sigma$ such that $(1-\theta)X\in T_o\Sigma$ (after
conjugation by an element of $Q$ if necessary). Then, if $X$ is
regular, Lemma~\ref{lemma:normalizes} implies
$[\g{h}_{\g{k}_0},X]\subset\g{g}_\alpha\ominus\g{w}$. Since the set
of regular vectors is dense, $X$ can always be approximated by a
sequence of regular vectors, and hence, by continuity we also obtain
$[\g{h}_{\g{k}_0},X]\subset\g{g}_\alpha\ominus\g{w}$ for non-regular
vectors. Therefore,
$[\g{h}_{\g{k}_0},\g{g}_\alpha\ominus\g{w}]\subset\g{g}_\alpha\ominus\g{w}$.
Finally, the skew-symmetry of the elements of
$\ad(\g{k}_0)$ implies
$[\g{h}_{\g{k}_0},\g{w}]\subset\g{w}$.
\end{proof}

We can now finish the proof of Proposition~\ref{prop:orbit equivalent}.

\begin{proof}[Proof of Proposition~\ref{prop:orbit equivalent}]
The fact that $\hat{\g{h}}=\g{q}\oplus\g{b}\oplus\g{w}\oplus\g{r}$ is
a subalgebra of $\g{k}_0\oplus\g{a}\oplus\g{n}$, and
Lemma~\ref{lemma:normalizes2}, imply that
$\tilde{\g{h}}=\g{h}_{\g{k}_0}\oplus\g{b}\oplus\g{w}\oplus\g{r}$ is a
Lie subalgebra of $\g{g}$ that contains $\g{h}$ and $\hat{\g{h}}$.
Let  $\tilde{H}$ be the connected subgroup of $G$ whose Lie algebra
is $\tilde{\g{h}}$. Since $T_o(H\cdot o)=T_o(\tilde{H}\cdot
o)=T_o(\hat{H}\cdot
o)=\g{b}\oplus(1-\theta)\g{w}\oplus(1-\theta)\g{r}$ and
$H\subset\tilde{H}$, $\hat{H}\subset \tilde{H}$, the orbits through
$o$ of the groups $H$, $\tilde{H}$, and $\hat{H}$ coincide. The principal orbits of the slice
representation of $H$ at $o$ are orbits of the slice representation of $\tilde{H}$ at $o$. Indeed, for a section $\Sigma$ through $o$ and
$v=aB+(1-\theta)\xi+x(1-\theta)Z\in T_o\Sigma$ with $\xi\in\g{s}$
regular, Lemma~\ref{lemma:normalizes} implies
$[\g{h}_{\g{k}_0},\xi]=\g{g}_\alpha\ominus(\g{w}\oplus\g{s})=[\g{q},\xi]$.
Thus, the tangent spaces at $v$ to the orbits of the slice
representations of $H$ and $\tilde{H}$ through $v$ coincide, and
since $H\subset\tilde{H}$, both orbits coincide. Then, Lemma~\ref{lemma:sameslice} guarantees that
the actions of $H$ and $\tilde{H}$ on $\C H^n$ have the same orbits.
Similarly, an analogous argument with $\hat{H}$ instead of $H$ allows
to show that the actions of $\hat{H}$ and $\tilde{H}$ on $\C H^n$ have the same orbits, and this completes the proof.
\end{proof}

We now proceed with the proof of Theorem~\ref{th:Parabolic}.

Let $H$ be a closed subgroup of the isometry group of $\C H^n$ acting
polarly on $\C H^n$, and assume that the Lie algebra of $H$ is
contained in a maximal parabolic subalgebra
$\g{k}_0\oplus\g{a}\oplus\g{n}$. As we argued at the beginning of
this section, there is a maximal compact subgroup $Q$ of $H$, and we
can assume that $o\in\C H^n$ is a fixed point of $Q$, that is, the
isotropy group of $H$ at $o$ is $Q$. We are now interested in
$\g{h}_{\g{a}\oplus\g{n}}$, the orthogonal  projection of $\g{h}$ on
$\g{a}\oplus\g{n}$. It is clear that $\g{h}_{\g{a}\oplus\g{n}}$ can
be written in one of the following forms: $\g{w}$,
$\R(B+X)\oplus\g{w}$, $\R(B+X+xZ)\oplus\g{w}$ (with $x\neq 0$),
$\g{w}\oplus\R(Y+Z)$, or $\R(B+X)\oplus \g{w}\oplus \R(Y+Z)$, where
$\g{w}\subset \g{g}_\alpha$, and $X$, $Y\in\g{g}_\alpha$.

In order to conclude the proof of Theorem~\ref{th:Parabolic} we deal
with these five possibilities separately.

\subsubsection*{{Case~1:}
$\g{h}_{\g{a}\oplus\g{n}}=\g{w}$, with $\g{w}$ a subspace of
$\g{g}_\alpha$}\hfill

Here $\g{h}$ is in the hypotheses of
Proposition~\ref{prop:orbit equivalent}, and it readily follows from
Lemma~\ref{lemma:section1} that this case is not possible.

\subsubsection*{{Case~2:} $\g{h}_{\g{a}\oplus\g{n}}=\R(B+X)\oplus\g{w}$,
with $\g{w}$ a subspace of $\g{g}_\alpha$, and
$X\in\g{g}_\alpha\ominus\g{w}$}\hfill

Assume first that $X\neq 0$. Then, $\nu_o(H\cdot o)=\R (-\lVert
X\rVert^2B+(1-\theta)X)\oplus(1-\theta)(\g{g}_\alpha\ominus\g{w})\oplus\g{p}_{2\alpha}$.
Let $\Sigma$ be a section through $o$. Since $T_o\Sigma\subset
\nu_o(H\cdot o)$, $[\g{q},-\lVert
X\rVert^2B+(1-\theta)X]\subset\g{p}_\alpha$,
$[\g{q},\g{p}_{2\alpha}]=0$, and
$[\g{q},\g{p}_\alpha]\subset\g{p}_\alpha$, we get that
$[\g{q},T_o\Sigma]$ is orthogonal to $\g{a}$ and $\g{p}_{2\alpha}$.
As $\nu_o(H\cdot o)=T_o\Sigma\oplus[\g{q},T_o\Sigma]$ (orthogonal
direct sum) by Proposition~\ref{prop:criterion2}, we readily get that
$\g{p}_{2\alpha}\subset T_o\Sigma$. Moreover, let
$T\in\g{h}_{\g{k}_0}$ be such that $T+B+X\in\g{h}$; then $T+B+X$ is
orthogonal to $[\g{q},T_o\Sigma]$, and since
$[\g{q},T_o\Sigma]\subset\g{p}_\alpha$ we obtain that $X$ is
orthogonal to $[\g{q},T_o\Sigma]$. The fact that the direct sum
$\nu_o(H\cdot o)=T_o\Sigma\oplus[\g{q},T_o\Sigma]$ is orthogonal
implies that $-\lVert X\rVert^2B+(1-\theta)X\in T_o\Sigma$. However,
since $T_o\Sigma$ is totally real we have
\[
0=\langle i(-\lVert
X\rVert^2B+(1-\theta)X),(1-\theta)Z\rangle=\langle -\frac{1}{2}\lVert
X\rVert^2(1-\theta)Z+(1-\theta)JX,(1-\theta)Z\rangle=-2\lVert
X\rVert^2,
\]
which is not possible because $X\neq 0$.

Therefore we must have $X=0$, and thus
$\g{h}_{\g{a}\oplus\g{n}}=\g{a}\oplus\g{w}$. Note that the fact that
$\g{h}$ is a subalgebra of $\g{k}_0\oplus\g{a}\oplus\g{n}$ implies
that $\g{w}$ is a totally real subspace of $\g{g}_\alpha$. We are now
in the hypotheses of Proposition~\ref{prop:orbit equivalent} and, as
shown in the proof of Lemma~\ref{lemma:normalizes}, $\g{w}=0$. We
conclude that the action of $H$ is orbit equivalent to the action of
the group $\hat{H}$ whose Lie algebra is
$\hat{\g{h}}=\g{q}\oplus\g{a}$. This corresponds to
Theorem~\ref{th:Parabolic}(\ref{th:Parabolic:qaw}).

\subsubsection*{{Case~3:}
$\g{h}_{\g{a}\oplus\g{n}}=\R(B+X+xZ)\oplus\g{w}$, with $\g{w}$ a
subspace of $\g{g}_\alpha$, $X\in\g{g}_\alpha\ominus\g{w}$, and
$x\in\R$, $x\neq 0$}\hfill

Let $g=\Exp(xZ)\in G$, and let $T+r(B+X+xZ)+V$ be a generic element
of $\g{h}$, with $V\in\g{w}$, $r\in\R$. Clearly, since $g\in AN$ we
have $\Ad(g)(\g{h})\subset\g{k}_0\oplus\g{a}\oplus\g{n}$. Then, it is
easy to obtain
\[
\Ad(g)(T+r(B+X+xZ)+V)=T+r(B+X+xZ)+V-rxZ=T+r(B+X)+V.
\]
Hence $(\Ad(g)(\g{h}))_{\g{a}\oplus\g{n}}=\R(B+X)\oplus\g{w}$, and
$\Ad(g)(\g{q})=\g{q}$. Since $Q$ is a maximal compact subgroup of
$I_g(H)=gHg^{-1}$, and the orthogonal projection of the Lie algebra
of $I_g(H)$ onto $\g{a}\oplus\g{n}$ is $\R(B+X)\oplus\g{w}$, the new
group $I_g(H)$ satisfies the conditions of Case~2. Therefore, the
action of $H$ is orbit equivalent to the action of the group
$\hat{H}$ whose Lie algebra is
$\hat{\g{h}}=\g{q}\oplus\g{a}$. This also corresponds to
Theorem~\ref{th:Parabolic}(\ref{th:Parabolic:qaw}).

\subsubsection*{{Case~4:} $\g{h}_{\g{a}\oplus\g{n}}=\g{w}\oplus\R(Y+Z)$,
with $\g{w}$ a subspace of $\g{g}_\alpha$, and
$Y\in\g{g}_\alpha\ominus\g{w}$}\hfill

Assume that $Y\neq 0$. Then, $\nu_o(H\cdot o)=\g{a}
\oplus(1-\theta)(\g{g}_\alpha\ominus\g{w})\oplus\R(2(1-\theta)Y-\lVert
Y\rVert^2(1-\theta)Z)$. Let $\Sigma$ be a section through $o$. Then,
by Proposition~\ref{prop:criterion2} we have $\nu_o(H\cdot
o)=T_o\Sigma\oplus[\g{q},T_o\Sigma]$ (orthogonal direct sum). Since
$[\g{q},2(1-\theta)Y-\lVert
Y\rVert^2(1-\theta)Z]\subset\g{p}_\alpha$, $[\g{q},\g{a}]=0$, and
$[\g{q},\g{p}_\alpha]\subset\g{p}_\alpha$, we get that
$[\g{q},T_o\Sigma]$ is orthogonal to $\g{a}$ and $\g{p}_{2\alpha}$.
Then, $\g{a}\subset T_o\Sigma$. On the other hand, if
$T\in\g{h}_{\g{k}_0}$ is such that $T+Y+Z\in\g{h}$, then $T+Y+Z$ is
orthogonal to $[\g{q},T_o\Sigma]\subset\nu_o(H\cdot o)$, and since
$[\g{q},T_o\Sigma]\subset\g{p}_\alpha$ we also obtain that $Y$ is
orthogonal to $[\g{q},T_o\Sigma]$. Thus, $2(1-\theta)Y-\lVert
Y\rVert^2(1-\theta)Z\in T_o\Sigma$. But, since $T_o\Sigma$ is totally
real, we get
\[
0=\langle B,i(2(1-\theta)Y-\lVert Y\rVert^2(1-\theta)Z)\rangle
=\langle B,2(1-\theta)JY+2\lVert Y\rVert^2 B\rangle=2\lVert Y\rVert^2,
\]
which contradicts $Y\neq 0$.

Therefore we have $Y=0$, and thus,
$\g{h}_{\g{a}\oplus\g{n}}=\g{w}\oplus\g{g}_{2\alpha}$. We are now in
the hypotheses of Proposition~\ref{prop:orbit equivalent}, and we
conclude that the action of $H$ is orbit equivalent to the action of
the connected subgroup $\hat{H}$ of the isometry group of $\C H^n$ whose
Lie algebra is $\hat{\g{h}}=\g{q}\oplus\g{w}\oplus\g{g}_{2\alpha}$,
with $\g{w}$ a subspace of $\g{g}_\alpha$. This corresponds to
Theorem~\ref{th:Parabolic}(\ref{th:Parabolic:qwg}).

\subsubsection*{{Case~5:} $\g{h}_{\g{a}\oplus\g{n}}
=\R(B+X)\oplus\g{w}\oplus\R(Y+Z)$,
with $\g{w}\subset\g{g}_\alpha$, and $X$,
$Y\in\g{g}_\alpha\ominus\g{w}$}\hfill

This final possibility is more involved.

Our first aim is to show that $Y=0$. So, assume for the moment that
$Y\neq 0$.

\begin{lemma}
We have $X=\gamma Y+\frac{2}{\lVert Y\rVert^2}JY$, with
$\gamma\in\R$.

\end{lemma}

\begin{proof}
Assume that $X$ and $Y$ are linearly dependent, that is, $X=\lambda
Y$, with $\lambda\in\R$. Then, $\g{h}_{\g{a}\oplus\g{n}}=\R(B+\lambda
Y)\oplus\g{w}\oplus\R(Y+Z)$, and there exist $T$,
$S\in\g{h}_{\g{k}_0}$ such that $T+B+\lambda Y$, $S+Y+Z\in\g{h}$.
Then,
\[
[T,S]+[T,Y]-\lambda[S,Y]+\frac{1}{2}Y+Z=[T+B+\lambda Y,S+Y+Z]\in\g{h}.
\]
Since $[T,Y]-\lambda[S,Y]\in\g{g}_\alpha\ominus\R Y$ by the
skew-symmetry of the elements of $\ad(\g{k}_0)$, we get
$\frac{1}{2}Y+Z\in\g{h}_{\g{a}\oplus\g{n}}$, which is not possible.

Therefore, we can assume that $X$ and $Y$ are linearly independent
vectors of $\g{g}_\alpha$. In particular, $X\neq 0$. Take and fix for
the rest of the calculations $T$, $S\in\g{h}_{\g{k}_0}$ such that
$T+B+X$, $S+Y+Z\in\g{h}$.

In this case, the normal space to the orbit through the origin $o$
can be written as
\begin{align*}
\nu_o(H\cdot o)={}
&\R(-\lVert X\rVert^2B+(1-\theta)X-\frac{1}{2}\langle X,Y\rangle(1-\theta)Z)
\oplus(\g{p}_\alpha\ominus(1-\theta)(\g{w}\oplus\R X\oplus\R Y))\\[-1ex]
&\oplus\R(-\langle X,Y\rangle B+(1-\theta)Y-\frac{1}{2}\lVert Y\rVert^2(1-\theta)Z).
\end{align*}
Let $\Sigma$ be a section of the action of $H$ on $\C H^n$ through
the point $o\in\C H^n$. By Proposition~\ref{prop:criterion2} we have
$\nu_o(H\cdot o)=T_o\Sigma\oplus[\g{q},T_o\Sigma]$ (orthogonal direct
sum).  In particular the vectors $T+B+X$ and $S+Y+Z$ are orthogonal
to $[\g{q},T_o\Sigma]\subset\g{p}_\alpha$ (because
$[\g{k}_0,\g{a}]=[\g{k}_0,\g{g}_{2\alpha}]=0$). This implies that $X$
and $Y$ are already orthogonal to $[\g{q},T_o\Sigma]$, and thus, so
are $-\lVert X\rVert^2B+(1-\theta)X-\frac{1}{2}\langle
X,Y\rangle(1-\theta)Z$ and $-\langle X,Y\rangle
B+(1-\theta)Y-\frac{1}{2}\lVert Y\rVert^2(1-\theta)Z$. Hence, they
are in $T_o\Sigma$ and we can write
\begin{align*}
T_o\Sigma={}
&\R(-\lVert X\rVert^2B+(1-\theta)X-\frac{1}{2}\langle
X,Y\rangle(1-\theta)Z)\\[-1ex]
&\oplus(1-\theta)\g{s} \oplus\R(-\langle
X,Y\rangle B+(1-\theta)Y-\frac{1}{2}\lVert Y\rVert^2(1-\theta)Z),
\end{align*}
where $\g{s}\subset\g{g}_\alpha\ominus\g{w}$ is totally real, and $\C X\oplus\C Y$
is orthogonal to $\g{s}$ (because sections are totally real). The
fact that $T_o\Sigma$ is totally real also implies
\begin{equation}\label{eq:real condition}
\begin{aligned}
0&{}=
\langle i(-\lVert X\rVert^2B+(1-\theta)(X-\frac{1}{2}\langle
X,Y\rangle Z)),-\langle
X,Y\rangle B+(1-\theta)(Y-\frac{1}{2}\lVert Y\rVert^2 Z)\rangle\\[-1ex]
&{}=\langle (1-\theta)(-\frac{1}{2}\lVert X\rVert^2 Z+JX)+\langle
X,Y\rangle B),-\langle
X,Y\rangle B+(1-\theta)(Y-\frac{1}{2}\lVert Y\rVert^2 Z)\rangle\\
&{}=\lVert X\rVert^2\lVert Y\rVert^2-\langle X,Y\rangle^2+2\langle JX,Y\rangle.
\end{aligned}
\end{equation}

Now, using Lemma~\ref{lemma:aux}(\ref{lemma:aux:a}), and
\eqref{eq:real condition}, we compute
\begin{align*}
&[-\lVert X\rVert^2B+(1-\theta)(X-\frac{1}{2}\langle
X,Y\rangle Z),-\langle
X,Y\rangle B+(1-\theta)(Y-\frac{1}{2}\lVert Y\rVert^2 Z)]\\[-1ex]
&\qquad=\frac{1}{2}(1+\theta)\bigl(-2[\theta X,Y]+\langle X,Y \rangle X
-\lVert X\rVert^2 Y-\lVert Y\rVert^2 JX
+\langle X,Y\rangle JY-\langle JX,Y\rangle Z\bigr).
\end{align*}
This vector is in $[T_o\Sigma,T_o\Sigma]$, which is orthogonal to
$\g{h}$ by Proposition~\ref{prop:criterion2}, so taking inner product
with $S+Y+Z$, and using Lemma~\ref{lemma:aux}(\ref{lemma:aux:b}) and
\eqref{eq:real condition}, we get $0=-2\langle
[S,X],Y\rangle-\frac{1}{2}\lVert Y\rVert^2\langle JX,Y\rangle$, which
implies
\begin{equation}\label{eq:sxy}
\langle [S,X],Y\rangle=-\frac{1}{4}\lVert Y\rVert^2\langle JX,Y\rangle.
\end{equation}

We also have
\[
[T+B+X,S+Y+Z]=[T,S]+[T,Y]-[S,X]+\frac{1}{2}Y+\bigl(1+\frac{1}{2}\langle JX,Y\rangle\bigr)Z,
\]
which is in $\g{h}$, so taking inner product with $-\langle
X,Y\rangle B+(1-\theta)(Y-\frac{1}{2}\lVert Y\rVert^2 Z)$, and
using~\eqref{eq:sxy}, we obtain
\[
0=-\langle [S,X],Y\rangle+\frac{1}{2}\lVert Y\rVert^2
-\lVert Y\rVert^2\bigl(1+\frac{1}{2}\langle JX,Y\rangle\bigr)
=-\frac{1}{2}\lVert Y\rVert^2\bigl(1+\frac{1}{2}\langle JX,Y\rangle\bigr).
\]
Since $Y\neq 0$, we get $\langle JX,Y\rangle=-2$ and
thus~\eqref{eq:real condition} can be written as
\[
\lVert X\rVert^2\lVert Y\rVert^2-\langle X,Y\rangle^2=4=\langle JX,Y\rangle^2.
\]
Now put $X=\gamma Y+\delta JY+E$ with $E$ orthogonal to $\C Y$,  and
$\gamma,\delta\in\R$. Then, the previous equation reads $\lVert
E\rVert^2\lVert Y\rVert^2=0$, which yields $E=0$. This implies the
result.
\end{proof}

Therefore the situation now is $\g{h}_{\g{a}\oplus\g{n}}=\R(B+\gamma
Y+\frac{2}{\lVert Y\rVert^2}JY)\oplus\g{w}\oplus\R(Y+Z)$, with  $\C
Y\subset\g{g}_\alpha\ominus\g{w}$. The normal space can be rewritten
as
\begin{align*}
\nu_o(H\cdot o)={}
&\R(-2B+(1-\theta)JY)
\oplus(\g{p}_\alpha\ominus(1-\theta)(\g{w}\oplus\C Y))\\[-1ex]
&\oplus\R(-\gamma\lVert Y\rVert^2 B+(1-\theta)Y-\frac{1}{2}\lVert Y\rVert^2(1-\theta)Z),
\end{align*}
and arguing as above, if $\Sigma$ is a section through $o$, then
\begin{equation}\label{eq:section}
T_o\Sigma=\R(-2B+(1-\theta)JY)
\oplus(1-\theta)\g{s} \oplus
\R(-\gamma\lVert Y\rVert^2 B+(1-\theta)Y-\frac{1}{2}\lVert Y\rVert^2(1-\theta)Z),
\end{equation}
where $\g{s}\subset\g{g}_\alpha\ominus(\g{w}\oplus\C Y)$ is a totally real
subspace of $\g{g}_\alpha$.

\begin{lemma}\label{lemma:sjy}
If $S\in\g{h}_{\g{k}_0}$ is such that $S+Y+Z\in\g{h}$ then
$[S,JY]=\frac{1}{4}\lVert Y\rVert^2 Y$.
\end{lemma}

\begin{proof}
First of all, by the properties of root systems and the skew-symmetry
of the elements of $\ad(\g{k}_0)$, we have
$[S,JY]\in\g{g}_\alpha\ominus\R JY$.

Lemma~\ref{lemma:aux}(\ref{lemma:aux:a}) yields
\begin{equation}\label{eq:bracket}
\begin{aligned}
&[-2B+(1-\theta)JY,-\gamma\lVert Y\rVert^2 B
+(1-\theta)(Y-\frac{1}{2}\lVert Y\rVert^2 Z)]\\[-1ex]
&\qquad=(1+\theta)\Bigl(-[\theta JY,Y]
+\Bigl(\frac{1}{2}\lVert Y\rVert^2-1\Bigr)Y+\frac{\gamma}{2}\lVert Y\rVert^2 JY
+\frac{1}{2}\lVert Y\rVert^2 Z\Bigr),
\end{aligned}
\end{equation}
which is a vector in $[T_o\Sigma,T_o\Sigma]$.

Take $U\in\g{w}$, and let $T_U\in\g{h}_{\g{k}_0}$ be such that
$T_U+U\in\g{h}$. Taking inner product with~\eqref{eq:bracket} and
using Lemma~\ref{lemma:aux}(\ref{lemma:aux:b}) we get $0=2\langle
[T_U,JY],Y\rangle$. Using this equality and since $\g{h}$ is a Lie subalgebra, we now have
\[
0=\langle [S+Y+Z,T_U+U],-2B+(1-\theta)JY\rangle
=\langle [S,T_U]+[S,U]-[T_U,Y],JY\rangle
=\langle[S,U],JY\rangle,
\]
and since $U\in\g{w}$ is arbitrary,
$[S,JY]\in\g{g}_\alpha\ominus(\g{w}\oplus\R JY)$.

Let $\xi\in\g{s}$. Proposition~\ref{prop:criterion2} implies
\[
0=\langle S+Y+Z,[-2B+(1-\theta)JY,(1-\theta)\xi]\rangle
=-\langle S,(1+\theta)[\theta JY,\xi]\rangle
=-2\langle[S,JY],\xi\rangle.
\]
Let $\eta\in\g{g}_\alpha\ominus(\g{w}\oplus\C Y)$ be an arbitrary vector.
Since $\Ad(Q)(T_o\Sigma)=\nu_o(H\cdot o)$ by
Proposition~\ref{prop:criterion2}, we can conjugate the section
$\Sigma$ in such a way that $\eta\in\g{s}$. (Note that
$-2B+(1-\theta)JY$ and $-\gamma\lVert Y\rVert^2
B+(1-\theta)Y-\frac{1}{2}\lVert Y\rVert^2(1-\theta)Z$ always belong
to $T_o\Sigma$ by~\eqref{eq:section}.) Hence, the equation above
shows that $[S,JY]$ is orthogonal to
$\g{g}_\alpha\ominus(\g{w}\oplus\C Y)$. Altogether this implies
$[S,JY]\in\R Y$.

Finally, taking inner product of~\eqref{eq:bracket} with
$S+Y+Z\in\g{h}$ we get, using
Lemma~\ref{lemma:aux}(\ref{lemma:aux:a}),
$0=2\langle[S,Y],JY\rangle+\frac{1}{2}\lVert Y\rVert^4$, and hence
$[S,JY]=\frac{1}{4}\lVert Y\rVert^2 Y$ as we wanted.
\end{proof}

We define $g=\Exp(-4JY/\lVert Y\rVert^2)$. Recall that the Lie algebra
of the isotropy group of $H$ at $g(o)$ is
$\g{h}_{g(o)}=\Ad(g)(\g{k})\cap\g{h}=\g{q}\cap\ker\ad(JY)$, according
to Lemma~\ref{lemma:isotropy}. Let
$S\in\g{h}_{\g{k}_0}$ be such that $S+Y+Z\in\g{h}$. Then,
Lemma~\ref{lemma:sjy} yields
\[
\Ad(g)(S)=S-\frac{4}{\lVert Y\rVert^2}[JY,S]+\frac{8}{\lVert Y\rVert^4}[JY,[JY,S]]
=S+Y+Z\in\Ad(g)(\g{k})\cap\g{h}.
\]
However, it is clear that $S+Y+Z\not\in\g{q}\cap\ker\ad(JY)$, which
gives a contradiction.

Therefore we have proved that $Y=0$. Thus
$\g{h}_{\g{a}\oplus\g{n}}=\R(B+X)\oplus\g{w}\oplus\g{g}_{2\alpha}$.
If $X=0$ then
$\g{h}_{\g{a}\oplus\g{n}}=\g{a}\oplus\g{w}\oplus\g{g}_{2\alpha}$, and
we are under the hypotheses of Proposition~\ref{prop:orbit
equivalent}, which implies that the action of $H$ is orbit equivalent
to the action of the group $\hat{H}$ whose Lie algebra is
$\hat{\g{h}}=\g{q}\oplus\g{a}\oplus\g{w}\oplus\g{g}_{2\alpha}$. This
corresponds to Theorem~\ref{th:Parabolic}(\ref{th:Parabolic:qawg}).

For the rest of this case we assume $X\neq 0$. Note that the normal
space to  the orbit through $o$ is $\nu_o(H\cdot o)=\R(-\lVert
X\rVert^2B+(1-\theta)X)\oplus(\g{p}_\alpha\ominus(1-\theta)(\g{w}\oplus\R
X))$. If $\Sigma$ is a section through $o$, since $\nu_o(H\cdot
o)=T_o\Sigma\oplus[\g{q},T_o\Sigma]$ (orthogonal direct sum), and
$[\g{q},T_o\Sigma]\subset\g{p}_\alpha$, it is easy to deduce, as in
previous cases, that
\[
T_o\Sigma=\R(-\lVert X\rVert^2B+(1-\theta)X)\oplus(1-\theta)\g{s},
\]
where $\R X\oplus\g{s}$ is a real subspace of $\g{g}_\alpha$.

We define $g=\Exp(2X)$. We will show
$(\Ad(g)(\g{h}))_{\g{a}\oplus\g{n}}=\g{a}\oplus\g{w}\oplus\g{g}_{2\alpha}$
and $\Ad(g)(\g{q})=\g{q}$, which will allow us to apply
Proposition~\ref{prop:orbit equivalent}. From now on we take
$T\in\g{h}_{\g{k}_0}$ such that $T+B+X\in\g{h}$.

Let $S\in\g{q}$. Then $[S,T]+[S,X]=[S,T+B+X]\in\g{h}$, and thus
$[S,X]\in\g{w}$.  Now let $U\in\g{w}$ be an arbitrary vector, and let
$S_U\in\g{h}_{\g{k}_0}$ such that $S_U+U\in\g{h}$. We have
$0=\langle[S,S_U+U],-\lVert
X\rVert^2B+(1-\theta)X\rangle=-\langle[S,X],U\rangle$, which together
with the previous assertion implies $[S,X]=0$. Then
$\Ad(g)(\g{q})=\g{q}$. In particular this implies that $Q$ is a
maximal compact subgroup of $I_g(H)=gHg^{-1}$.

Now we calculate $[T,X]$. Let $U\in\g{w}$ and $S_U\in\g{h}_{\g{k}_0}$
such that $S_U+U\in\g{h}$. Then, by the skew-symmetry of the elements
of $\ad(\g{k}_0)$ we have $0=\langle [T+B+X,S_U+U],-\lVert
X\rVert^2B+(1-\theta)X\rangle=-\langle[T,X],U\rangle$, so
$[T,X]\in\g{g}_\alpha\ominus\g{w}$. Let now $\xi\in\g{s}$. By
Proposition~\ref{prop:criterion2} we get, using
Lemma~\ref{lemma:aux}(\ref{lemma:aux:b}), $0=\langle T+B+X,[-\lVert
X\rVert^2B+(1-\theta)X,(1-\theta)\xi]\rangle=-\langle
T,(1+\theta)[\theta X,\xi]\rangle=-2\langle[T,X],\xi\rangle$. Using
again Proposition~\ref{prop:criterion2} we have $\nu_o(H\cdot
o)=\Ad(Q)(T_o\Sigma)$, and thus, for any
$\eta\in\g{g}_\alpha\ominus(\g{w}\oplus\R X)$ we can find a section
through $o$ such that $(1-\theta)\eta\in T_o\Sigma$ (note that
$-\lVert X\rVert^2B+(1-\theta)X\in T_o\Sigma$ for any section). Hence
the previous argument shows $\langle[T,X],\eta\rangle=0$, and
altogether this means $[T,X]=0$. Therefore, $\Ad(g)(T+B+X)=T+B$, so
the projection of this vector onto $\g{a}\oplus\g{n}$ is in
$\g{a}\subset\g{a}\oplus\g{w}\oplus\g{g}_{2\alpha}$.

Fix $U\in\g{w}$ and $S_U\in\g{h}_{\g{k}_0}$ such that
$S_U+U\in\g{h}$.  We calculate $[S_U,X]$. For any $\xi\in\g{s}$, by
Proposition~\ref{prop:criterion2} and
Lemma~\ref{lemma:aux}(\ref{lemma:aux:b}), we get $0=\langle
S_U+U,[-\lVert
X\rVert^2B+(1-\theta)X,(1-\theta)\xi]\rangle=-2\langle[S_U,X],\xi\rangle$.
As in the previous paragraph, one can argue that $\xi$ can be taken
arbitrarily in $\g{g}_\alpha\ominus(\g{w}\oplus\R X)$ by changing the
tangent space to the section, if necessary, by an element of
$\Ad(Q)$. Hence $[S_U,X]\in\g{w}$, which yields
$\Ad(g)(S_U+U)=S_U+U-2[S_U,X]+\frac{1}{2}(\langle JX, U\rangle -
2\langle JX,[S_U,X]\rangle)Z$,  and thus, its projection onto
$\g{a}\oplus\g{n}$ belongs to
$\g{a}\oplus\g{w}\oplus\g{g}_{2\alpha}$.

Finally, let $S_Z\in\g{h}_{\g{k}_0}$ such that $S_Z+Z\in\g{h}$. For
each $\xi\in\g{s}$ we obtain $0=\langle S_Z+Z,[-\lVert
X\rVert^2B+(1-\theta)X,(1-\theta)\xi]\rangle=-2\langle[S_Z,X],\xi\rangle$,
and since $\xi$ can be taken to be in
$\g{g}_\alpha\ominus(\g{w}\oplus\R X)$ by a suitable conjugation of
the section by an element in $\Ad(Q)$, we deduce $[S_Z,X]\in\g{w}$. Hence,
$\Ad(g)(S_Z+Z)=S_Z-2[S_Z,X]+(1-\langle JX,[S_Z,X]\rangle)Z$, and the
orthogonal  projection of this vector onto $\g{a}\oplus\g{n}$ belongs
to $\g{a}\oplus\g{w}\oplus\g{g}_{2\alpha}$.

These last calculations show that
$(\Ad(g)(\g{h}))_{\g{a}\oplus\g{n}}\subset\g{a}\oplus\g{w}\oplus\g{g}_{2\alpha}$.
Since $g\in AN$ normalizes $\g{k}_0\oplus\g{a}\oplus\g{n}$, we have
that  $\Ad(g)(\g{h})\subset\g{k}_0\oplus\g{a}\oplus\g{n}$. Then the
kernel of the projection of $\Ad(g)(\g{h})$ onto $\g{a}\oplus\g{n}$
is precisely $\Ad(g)(\g{h})\cap \g{k}_0$, which is a compact
subalgebra of $\Ad(g)(\g{h})$ containing $\g{q}=\Ad(g)(\g{q})$. By
the maximality of $\g{q}$ we get that $\Ad(g)(\g{h})\cap
\g{k}_0=\g{q}$. But then by elementary linear algebra
\begin{align*}
\dim (\Ad(g)\g{h})_{\g{a}\oplus\g{n}}
&{}=\dim\Ad(g)(\g{h})-\dim(\Ad(g)(\g{h})\cap\g{k}_0)\\
&{}=\dim\g{h}-\dim \g{q}
=\dim\g{h}_{\g{a}\oplus\g{n}}=\dim (\g{a}\oplus\g{w}\oplus\g{g}_{2\alpha}).
\end{align*}

All in all we have shown that the Lie algebra $\Ad(g)(\g{h})$ of
$I_g(H)=gHg^{-1}$ satisfies
$(\Ad(g)(\g{h}))_{\g{a}\oplus\g{n}}=\g{a}\oplus\g{w}\oplus\g{g}_{2\alpha}$,
and that $Q$ is a maximal compact subgroup of $I_g(H)$. Therefore, we
can apply Proposition~\ref{prop:orbit equivalent} to $I_g(H)$. This
implies that the action of $H$ on $\C H^n$ is orbit equivalent to
the action of the group $\hat{H}$ whose Lie algebra is
$\hat{\g{h}}=\g{q}\oplus\g{a}\oplus\g{w}\oplus\g{g}_{2\alpha}$. This
corresponds to Theorem~\ref{th:Parabolic}(\ref{th:Parabolic:qawg}).

Altogether, we have concluded the proof of Theorem~\ref{th:Parabolic}.


\section{Proof of the main results}\label{sec:proof}

In this section we conclude the proof of
Theorems~A and~B
using the results of Sections~\ref{sec:totally_geodesic}
and~\ref{sec:parabolic}.

\begin{proof}[Proof of Theorem~A]
The actions described in part~\eqref{th:main:so} are polar by virtue
of Theorem~\ref{theorem:ProdAct} and Theorem~\ref{th:SubSO}, whereas
the polarity of the actions in part~\eqref{th:main:parabolic} follows
from Theorem~\ref{th:with_a_general}.

An action of a subgroup $H$ of the isometry group $I(M)$ of a
Riemannian manifold~$M$ is proper if and only if $H$ is a closed
subgroup of~$I(M)$. Hence we may assume $H \subset SU(1,n)$ is
closed. Since the polarity of the action depends only on the Lie
algebra of~$H$ by Proposition~\ref{prop:criterion1}, we may assume
that $H$ is connected.

Thus, let $H$ be a connected closed subgroup of $SU(1,n)$ acting
polarly on $\C H^n$. Using the Karpelevich-Mostow Theorem~\cite{Ka53}, \cite{Mo55}, it was proved in~\cite[Theorem~6.2]{ADS03} that $H$ fixes a unique point in the ideal boundary of $\C H^n$, or $H$ has a totally geodesic orbit. The first possibility is equivalent to saying that $H$ is contained in a maximal parabolic subgroup of~$G$ (see for example~\cite[\S2.17]{E96}). In the second case, an orbit of $H$ is a totally geodesic $\C H^k$ or a totally geodesic $\R H^k$, $k\in\{0,\dots,n\}$.

Let $k$ be the smallest complex dimension of a totally geodesic
complex hyperbolic  subspace left invariant by the $H$-action. If
$k=0$, then the $H$-action has a fixed point. In this case, it
follows from~\cite{DK11} that $H$ is a subgroup of $S(U(1)U(n))\cong
U(n)$ that corresponds to a polar action on $\C P^{n-1}$, and
therefore is induced by the isotropy representation of a Hermitian
symmetric space. This corresponds to case~\eqref{th:main:so} with
$k=0$ in Theorem~A.

Let us assume from now on that $k\geq 1$. Theorem~\ref{theorem:ProdAct} guarantees that
the $H$-action is orbit equivalent to the product action of a closed
subgroup $H_1$ of $SU(1,k)$ acting polarly on $\C H^k$, times a closed
subgroup $H_2$ of $U(n-k)$ acting polarly (and with a fixed point)
on $\C H^{n-k}$. By assumption, the $H_1$-action on $\C H^k$ does not
leave any totally geodesic $\C H^{l}$ or $\R H^{l}$ with $l<k$
invariant. Hence, the $H_1$-action on $\C H^k$ is transitive, or has a totally geodesic $\R H^k$ as an orbit, or $H_1$ is contained
in a maximal parabolic subgroup of $SU(1,k)$. The first case corresponds to Theorem~A\eqref{th:main:parabolic} with $\g{q}=\g{h}_2$, $\g{b}=\g{a}$, and $\g{w}$ a complex subspace of $\g{g}_\alpha$ with dimension $k-1$. In the second case, we have shown in Theorem~\ref{th:SubSO} that the $H_1$-action on the totally geodesic $\C H^k$ is of cohomogeneity one and orbit equivalent to the action of $SO^0(1,k)$; thus, this
corresponds to part~\eqref{th:main:so} with $k\in\{1,\dots,n\}$. Note
that for $Q=H_2$, the $Q$-action on $\C H^{n-k}$ is determined by its
slice representation at the fixed point, so $Q$ acts polarly with a
totally real section on $T_o\C H^{n-k}\cong \C^{n-k}$.

Let us consider the final case, that is, $H_1$ is contained in a
maximal parabolic subgroup of $SU(1,k)$, $k\in\{1,\dots,n\}$. The Lie algebra $\g{su}(1,k)$ can be assumed to be of the form $\g{su}(1,k)=\g{k}^1\oplus\g{a}\oplus\g{g}_\alpha^1\oplus\g{g}_{2\alpha}$, where $\g{g}_\alpha^1$ is a
complex subspace of $\g{g}_\alpha$ with complex dimension $k-1$, and
$\g{k}^1=\g{k}\cap\g{su}(1,k)$. Indeed, $\g{k}^1\oplus\g{a}\oplus\g{g}_\alpha^1\oplus\g{g}_{2\alpha}$ is an Iwasawa decomposition of $\g{su}(1,k)\subset \g{su}(1,n)$. Now, a maximal parabolic subalgebra of $\g{su}(1,k)$ is of the form $\g{k}_0^1 \oplus \g{a} \oplus
\g{g}_\alpha^1\oplus\g{g}_{2\alpha}$, where $\g{k}_0^1\cong \g{u}(k-1)$ is the normalizer of $\g{a}$ in
$\g{k}^1$. Thus, assume that $\g{h}_1\subset\g{k}_0^1 \oplus \g{a} \oplus
\g{g}_\alpha^1\oplus\g{g}_{2\alpha}$. It follows that the $H_1$-action is orbit
equivalent by an element of $SU(1,k)$ to the action of a closed subgroup $H_1'$ of $SU(1,k)$ with one
of the Lie algebras described in Theorem~\ref{th:Parabolic}:
\eqref{th:Parabolic:qaw}~$\g{q}^1\oplus\g{a}$,
\eqref{th:Parabolic:qawg}
$\g{q}^1\oplus\g{a}\oplus\g{w}\oplus\g{g}_{2\alpha}$, or
\eqref{th:Parabolic:qwg} $\g{q}^1\oplus\g{w}\oplus\g{g}_{2\alpha}$,
where $\g{w}$ is a real subspace of $\g{g}_\alpha^1$, and
$\g{q}^1\subset\g{k}_0^1$ normalizes $\g{w}$. By conjugating by an element of $SU(1,k)$ we may assume that $H_1$ and $H_1'$ have the same orbits in $\C H^k$. Let $p\in\C H^k$ be an arbitrary point. Then, for each $h_1\in H_1$ there exists $h_1'\in H_1'$ such that $h_1(p)=h_1'(p)$, and thus, $h_1^{-1}h_1'$ is contained in the isotropy group of $SU(1,k)$ at $p$. This implies that for each $\eta\in\nu_p\C H^k$ we have $(h_1^{-1}h_1')_*\eta=\eta$. Then, for each $h_2\in H_2\subset U(n-k)$ we have that
\[
h_1h_2(\exp_p(\xi))=\exp_{h_1(p)}(h_{1*}h_{2*}\xi)=
\exp_{h_1'(p)}(h_{1*}'h_{2*}\xi)=h_1'h_2(\exp_p(\xi)),
\]
for each $p\in\C H^k$ and $\xi\in\nu_p\C H^k$. Hence, it follows that $H_1\times H_2$ and $H_1'\times H_2$ act on $\C H^n$ with the same orbits.

Finally we can define
$\g{q}=\g{q}^1\oplus\g{h}_2$, which is a subalgebra of $\g{k}_0$.
Therefore, case~\eqref{th:Parabolic:qaw} above corresponds to the Lie algebra $\g{q}\oplus\g{a}$, which is a
particular case of Theorem~A\eqref{th:main:so} for $k=1$, while
cases~\eqref{th:Parabolic:qawg} and \eqref{th:Parabolic:qwg} correspond to the Lie algebras $\g{q}\oplus\g{a}\oplus\g{w}\oplus\g{g}_{2\alpha}$ and $\g{q}\oplus\g{w}\oplus\g{g}_{2\alpha}$, which in turn correspond to
Theorem~A\eqref{th:main:parabolic}, where $\g{b}=\g{a}$ and
$\g{b}=0$, respectively. Lemma~\ref{lemma:equivariance},
Proposition~\ref{prop:realSection} and the fact that the slice
representation of a polar action is also polar, guarantee that the
action of $\g{q}$ on the orthogonal complement of $\g{w}$ in
$\g{g}_{\alpha}$ is polar with a totally real section.
\end{proof}

Before beginning the proof of Theorem~B, we need to calculate the
mean curvature vector of the orbits of minimum orbit type.

\begin{lemma}\label{lemma:meanCurvature}
Let $H$ be the connected Lie subgroup of $SU(1,n)$ whose Lie algebra
is $\g{h}=\R(aB+X)\oplus\g{w}\oplus\g{g}_{2\alpha}$, for some
$a\in\R$, $\g{w}$ subspace of $\g{g}_\alpha$, and
$X\in\g{g}_\alpha\ominus\g{w}$, $a\neq 0$, $X\neq 0$. Then, the mean
curvature vector of $H\cdot o$ is
\[
\mathcal{H}=\frac{3+\dim\g{w}}{2(a^2+\lVert X\rVert^2)}(\lVert X\rVert^2 B-a X).
\]
\end{lemma}

\begin{proof}
In order to shorten the notation, let us denote by
$\langle\,\cdot\,,\,\cdot\,\rangle$ the metric on $AN$ defined in
Section~\ref{sec:preliminaries}. Then, it is well known that the
Levi-Civita connection of $AN$ is given by (see for
example~\cite{BD09} or \cite{BTV95})
\[
{\nabla}_{aB+U+xZ}(bB+V+yZ)
=\left(\frac{1}{2}\langle U,V\rangle+xy\right)\!B-\frac{1}{2}\left(bU+yJU+xJV\right)
+\left(\frac{1}{2}\langle JU,V\rangle-bx\right)\!Z,
\]
where $a$, $b$, $x$, $y\in\R$, $U$, $V\in\g{g}_\alpha$, and all
vector fields are considered to be left-invariant. The normal space
to the orbit $H\cdot o$ is given by the left-invariant distribution
$\R\xi\oplus(\g{g}_\alpha\ominus(\g{w}\oplus\R X))$, where $\xi$ is
the unit vector
\[
\xi=\frac{1}{\lVert X\rVert\sqrt{a^2+\lVert X\rVert^2}}(\lVert X\rVert^2 B-a X).
\]
Then, the shape operator $\mathcal{S}_\eta$ with respect to a
left-invariant normal vector $\eta$ is given by the equation
$\mathcal{S}_\eta V=-(\nabla_V\eta)^\top$, where $(\cdot)^\top$ means
orthogonal projection onto $\g{h}$.

Bearing all this in mind and applying the formula for the
Levi-Civita connection above we get:
\begin{align*}
\mathcal{S}_\xi(aB+X) &=\frac{\lVert X\rVert}{2\sqrt{a^2+\lVert X\rVert^2}}(aB+X),\\
\mathcal{S}_\xi(W)    &=\frac{\lVert X\rVert}{2\sqrt{a^2+\lVert X\rVert^2}}
    \Bigl(W+a\frac{\langle JW,X\rangle}{\lVert X\rVert^2}Z\Bigr),
    \text{ for each $W\in\g{w}$},\\
\mathcal{S}_\xi(Z)    &=\frac{\lVert X\rVert}{2\sqrt{a^2+\lVert X\rVert^2}}
    \Bigl(-\frac{a}{\lVert X\rVert^2}(JX)^\top+2Z\Bigr).
\end{align*}
This implies
\[
\tr \mathcal{S}_\xi = \frac{(3+\dim\g{w})\lVert X\rVert}{2\sqrt{a^2+\lVert X\rVert^2}}.
\]
On the other hand, if $U\in\g{g}_\alpha\ominus(\g{w}\oplus\R X)$ we
have
\begin{align*}
\mathcal{S}_U(aB+X) &=\frac{1}{2}\langle JU,X\rangle Z,\\
\mathcal{S}_U(W)    &=\frac{1}{2}\langle JU,W\rangle Z,\text{ for each $W\in\g{w}$},\\
\mathcal{S}_U(Z)    &=\frac{1}{2}(JU)^\top.
\end{align*}
Hence $\tr \mathcal{S}_U = 0$.

Altogether we have proved the lemma.
\end{proof}

We can also obtain easily the following result.

\begin{corollary}\label{corol:meanCurvature}
If $\g{h}=\g{b}\oplus\g{w}\oplus\g{g}_{2\alpha}$, with
$\g{b}\in\{0,\g{a}\}$ and $\g{w}$ a subspace of $\g{g}_\alpha$, then
the mean curvature vector of $H\cdot o$ is
\[
\mathcal{H}=
\begin{cases}
0, &\text{if $\g{b}=\g{a}$},\\
\frac{1}{2}(2+\dim\g{w})B, &\text{if $\g{b}=0$.}
\end{cases}
\]
\end{corollary}

Now we finish the proof of Theorem~B.

\begin{proof}[Proof of Theorem~B]
In the following we consider some orbits of \emph{minimum orbit type}, as in Proposition~\ref{prop:MinimumOrbitType}; some of them are also \emph{minimal submanifolds} in the sense that their mean curvature vector vanishes. If this is the case, we will say that they have \emph{vanishing mean curvature} in order to avoid confusion.

Let $H_1$ and $H_2$ be two subgroups of $U(1,n)$ acting polarly on
$\C H^n$, and assume that these two actions are orbit equivalent. Let
us denote by $\g{h}_1$ and $\g{h}_2$ the Lie algebras of $H_1$ and
$H_2$. We distinguish three main cases.

\subsubsection*{{Case} \emph{1}} First of all, assume that the actions of $H_1$ and $H_2$ are given by
Theorem~A(\ref{th:main:so}), that is,
$\g{h}_i=\g{q}_i\oplus\g{so}(1,k_i)$, $i\in\{1,2\}$. The group $H_i$
has a totally real, totally geodesic $\R H^{k_i}$ as a singular orbit
of minimum orbit type.
This immediately implies $k_1=k_2$. If $k_1=k_2=n$ then both actions
are orbit equivalent to the action of $SO^0(1,n)$, according to
Theorem~\ref{th:SubSO}, and thus the isotropy actions of $Q_1$ and
$Q_2$ are both trivial. Assume $k_1=k_2<n$. It follows from
Section~\ref{sec:totally_geodesic} that $H_i$ restricts to a
cohomogeneity one action on the corresponding totally geodesic $\C
H^{k_i}$ that contains this $\R H^{k_i}$. It follows from
Theorem~\ref{theorem:ProdAct} that the slice representation of $H_i$
at $o$ is polar and a section of this action is of the form
$\ell_i\oplus\Sigma_i$, where $\ell_i$ is a line in $T_o\C H^{k_i}$
and $\Sigma_i$ is totally real in the complex subspace $T_o\C
H^n\ominus T_o\C H^{k_i}$. The unitary representation of
$Q_i$ on the complex vector space $T_o\C H^n\ominus T_o\C H^{k_i}$ cannot have trivial
factors, since otherwise a maximal trivial factor would be complex, and therefore the section would not be totally real. Hence, the only orbit of minimal dimension must be the
totally geodesic $\R H^{k_i}$. Since the actions of $H_1$ and $H_2$
are orbit equivalent, we conclude that $H_1\cdot o$ must be mapped to
$H_2\cdot o$. Now it is easy to deduce that the actions of $Q_1$ and
$Q_2$ must be orbit equivalent.

\subsubsection*{{Case} \emph{2}} Assume now that $H_1$ is given by
Theorem~A(\ref{th:main:so}), and $H_2$ is given by Theorem~A(\ref{th:main:parabolic}), that is,
$\g{h}_1=\g{q}_1\oplus\g{so}(1,k)$ and
$\g{h}_2=\g{q}_2\oplus\g{b}\oplus\g{w}\oplus\g{g}_{2\alpha}$. By
assumption, there is an isometry $\phi$ such that $\phi(H_1\cdot
p)=H_2\cdot\phi(p)$ for any $p\in\C H^n$. We know that $H_1\cdot o$ is a
totally real, totally geodesic $\R H^k$, and a singular orbit of
minimum orbit type. Let $g\in AN$ be such that $\phi(o)=g(o)$ and define
$\tilde{H}_2$ to be the connected Lie subgroup of $SU(1,n)$ whose Lie
algebra is $\tilde{\g{h}}_2=\g{b}\oplus\g{w}\oplus\g{g}_{2\alpha}$.
As $\phi(H_1\cdot o)$ must also be an orbit of $H_2$ of minimal
dimension, it follows that $\tilde{H}_2\cdot g(o)=H_2\cdot g(o)$ since
${H}_2\cdot g(o)$ has minimal dimension and contains $\tilde{H}_2\cdot g(o)$. We have $\tilde{H}_2\cdot g(o)=g(g^{-1}\tilde{H}_2g)\cdot
o=g(I_{g^{-1}}(\tilde{H}_2)\cdot o)$, from where it follows that
$\tilde{H}_2\cdot g(o)$ is congruent to the orbit
of $I_{g^{-1}}(\tilde{H}_2)$ through $o$. Since $g\in AN$, it is not difficult to
deduce from the bracket relations of $AN$ that
$\Ad(g^{-1})\tilde{\g{h}}_2=\R(aB+X)\oplus\g{w}\oplus\g{g}_{2\alpha}$
for some $a\in\R$ and $X\in\g{g}_\alpha\ominus\g{w}$. The fact that
$\Ad(g^{-1})\tilde{\g{h}}_2$ is totally real immediately implies
$a=0$, and thus $\g{b}=0$. Moreover, $X=0$ for dimension reasons, and $\g{w}$ is
totally real. Then,
Corollary~\ref{corol:meanCurvature} implies that
$\Ad(g^{-1})\tilde{\g{h}}_2=\g{w}\oplus\g{g}_{2\alpha}$ has mean
curvature vector $\mathcal{H}=\frac{1}{2}(2+\dim\g{w})B\neq 0$. In
particular, $H_2\cdot g(o)$ cannot be totally geodesic. Therefore,
polar actions given by Theorem~A(\ref{th:main:so}) cannot be
orbit equivalent to polar actions given by
Theorem~A(\ref{th:main:parabolic}).

\subsubsection*{{Case} \emph{3}} Finally, assume that $H_1$ and $H_2$ are
given by Theorem~A(\ref{th:main:parabolic}), that
is, $\g{h}_i=\g{q}_i\oplus\g{b}_i\oplus\g{w}_i\oplus\g{g}_{2\alpha}$,
$i\in\{1,2\}$.

Let $\g{b}_1=\g{a}$, $\g{b}_2=0$. The orbit $H_1\cdot o$ is of minimum orbit type, and it has vanishing mean curvature as Corollary~\ref{corol:meanCurvature} shows. This orbit
must be mapped to an orbit of $H_2$ of minimal dimension. Let
$\tilde{H}_2$ be the Lie subgroup of $SU(1,n)$ whose Lie algebra is
$\tilde{\g{h}}_2=\g{w}_2\oplus\g{g}_{2\alpha}$. Assume $H_1\cdot o$ is
mapped to $H_2\cdot g(o)$ with $g\in AN$. Since $H_2\cdot o$ has minimal dimension, we must have $H_2\cdot g(o)=\tilde{H}_2\cdot
g(o)$, as $H_2\cdot g(o)$ has minimal dimension and $\tilde{H}_2\cdot g(o)\subset H_2\cdot g(o)$. We have
$\tilde{H}_2\cdot g(o)=g(g^{-1}\tilde{H}_2 g\cdot o)$, and it is easy
to show using the bracket relations of $AN$ that
$\Ad(g^{-1})\tilde{\g{h}}_2=\tilde{\g{h}}_2=\g{w}_2\oplus\g{g}_{2\alpha}$.
Corollary~\ref{corol:meanCurvature} then implies that ${H}_2\cdot
g(o)=\tilde{H}_2\cdot g(o)$ has non-vanishing mean curvature. This contradicts the fact that the mean curvature vector of $H_1\cdot o$ is zero. Therefore
$\g{b}_1=\g{b}_2$.

Assume that $\g{b}_1=\g{b}_2=0$. Let $\phi$ be an isometry
such that $\phi(H_1\cdot o)=H_2\cdot \phi(o)$, and take $g\in AN$
such that $\phi(o)=g(o)$. Let $\tilde{H}_2$ be the Lie subgroup of
$SU(1,n)$ whose Lie algebra is
$\tilde{\g{h}}_2=\g{w}_2\oplus\g{g}_{2\alpha}$, and recall that
$\Ad(g^{-1})\tilde{\g{h}}_2=\tilde{\g{h}}_2$, and thus
$I_{g^{-1}}(\tilde{H}_2)=\tilde{H}_2$. Since $H_2\cdot g(o)$ must be
of minimal dimension, we have $H_2\cdot g(o)=\tilde{H}_2\cdot
g(o)=g(g^{-1}\tilde{H}_2g\cdot o)=g(\tilde{H}_2\cdot o)=g(H_2\cdot
o)$, and thus, $g^{-1}\phi(H_1\cdot
o)=g^{-1}H_2\cdot\phi(o)=g^{-1}H_2g\cdot o=H_2\cdot o$. By composing
with an element of $H_2$ we can further assume that $\phi(H_1\cdot
o)=H_2\cdot o$, and $\phi(o)=o$. In particular, we have
$\phi_*(T_o(H_1\cdot o))=T_o(H_2\cdot o)$, that is,
$\phi_*((1-\theta)(\g{w}_1\oplus\g{g}_{2\alpha}))
=(1-\theta)(\g{w}_2\oplus\g{g}_{2\alpha})$. We have the K\"{a}hler angle
decompositions (see Subsection~\ref{subsec:realsubspace}),
$\g{w}_i=\oplus_{\varphi\in\Phi_i}\g{w}_{i,\varphi}$, and thus,
$(1-\theta)\g{w}_i=\oplus_{\varphi\in\Phi_i}(1-\theta)\g{w}_{i,\varphi}$.
Since $\phi_*$ maps real subspaces of constant K\"{a}hler angle $\varphi$ to real subspaces of constant K\"{a}hler angle $\varphi$,
we must have $\Phi:=\Phi_1=\Phi_2$
and
\[
\phi_*(1-\theta)(\g{w}_{1,\pi/2}\oplus\g{g}_{2\alpha})
=(1-\theta)(\g{w}_{2,\pi/2}\oplus\g{g}_{2\alpha}),\quad
\phi_*(1-\theta)(\g{w}_{1,\varphi})
=(1-\theta)(\g{w}_{2,\varphi}),
\]
for all $\varphi\in\Phi\setminus\{\pi/2\}$.
As a consequence, $\dim\g{w}_{1,\varphi}=\dim\g{w}_{2,\varphi}$ for
all $\varphi\in\Phi$. It follows from Remark~\ref{rem:KahlerAngles}
that there exists $k\in K_0$ such that
$\Ad(k)\g{w}_{1,\varphi}=\g{w}_{2,\varphi}$ for all $\varphi\in\Phi$,
and thus
$\Ad(k)(\g{w}_1\oplus\g{g}_{2\alpha})=\g{w}_2\oplus\g{g}_{2\alpha}$.
Let $\hat{H}_2=k^{-1}H_2 k$. Obviously, the actions of $H_2$ and
$\hat{H}_2$ on $\C H^n$ are orbit equivalent. Indeed,
$\hat{\g{h}}_2=\hat{\g{q}}_2\oplus\g{w}_1\oplus\g{g}_{2\alpha}$ for
some subalgebra $\hat{\g{q}}_2$ of $\g{k}_0$. Since the actions of
$H_1$ and $\hat{H}_2$ are orbit equivalent, their slice
representations are orbit equivalent and so are the actions of $Q_1$
and $\hat{Q}_2$ on the normal space of $H_1\cdot o=\hat{H}_2\cdot o$ (note that the action of $K_0$ on $\g{a}$ is
trivial). Therefore, $H_1$ and $H_2$ are orbit equivalent if and only
if there exists $k\in K_0$ such that $\Ad(k)\g{w}_1=\g{w}_2$, and the
slice representations
$Q_i\times(1-\theta)\g{w}_i^\perp\to(1-\theta)\g{w}_i^\perp$,
$i\in\{1,2\}$, are orbit equivalent.

We now deal with the last possibility: $\g{b}_1=\g{b}_2=\g{a}$. The
proof goes along the lines of the previous subcase, with some
differences that we will point out. Let $\phi$ be an isometry such
that $\phi(H_1\cdot o)=H_2\cdot\phi(o)$ and take $g\in AN$ such that
$\phi(o)=g(o)$. As before, we consider $\tilde{H}_2$ the subgroup of
$SU(1,n)$ whose Lie algebra is
$\tilde{\g{h}}_2=\g{a}\oplus\g{w}_2\oplus\g{g}_{2\alpha}$. Since
$H_2\cdot g(o)$ is of minimal dimension, $H_2\cdot g(o)=\tilde{H}_2\cdot
g(o)=g(g^{-1}\tilde{H}_2g\cdot o)$, so $H_2\cdot g(o)$ is congruent
to the orbit through $o$ of the Lie subgroup of $SU(1,n)$
whose Lie algebra is of the form
$\Ad(g)\tilde{\g{h}}_2=\R(aB+X)\oplus\g{w}_2\oplus\g{g}_{2\alpha}$,
for some $a\in\R$ and $X\in\g{g}_\alpha\ominus\g{w}$. Since $H_1\cdot o$ has vanishing mean curvature by Corollary~\ref{corol:meanCurvature}, according to Lemma~\ref{lemma:meanCurvature} we must have that $X=0$.
In this case we get
$\Ad(g)\tilde{\g{h}}_2=\g{h}_2$. By composing with an element of $H_2$ we
can further assume that $\phi(o)=o$. Arguing as in the previous case,
we have the K\"{a}hler angle decompositions
$\g{w}_i=\oplus_{\varphi\in\Phi_i}\g{w}_{i,\varphi}$, and thus, it
follows that $\Phi:=\Phi_1=\Phi_2$ and
\[
\phi_*(1-\theta)(\g{a}\oplus\g{w}_{1,0}\oplus\g{g}_{2\alpha})
=(1-\theta)(\g{a}\oplus\g{w}_{2,0}\oplus\g{g}_{2\alpha}),\quad
\phi_*(1-\theta)(\g{w}_{1,\varphi})
=(1-\theta)\g{w}_{2,\varphi},
\]
for all $\varphi\in\Phi\setminus\{0\}$.
Again, by Remark~\ref{rem:KahlerAngles}, it follows that there exists
$k\in K_0$ such that $\Ad(k)\g{w}_{1,\varphi}=\g{w}_{2,\varphi}$ for
all $\varphi\in\Phi$, and therefore, $H_1$ and $H_2$ are orbit
equivalent if and only if there exists $k\in K_0$ such that
$\Ad(k)\g{w}_1=\g{w}_2$, and the slice representations
$Q_i\times(1-\theta)\g{w}_i^\perp\to(1-\theta)\g{w}_i^\perp$,
$i\in\{1,2\}$, are orbit equivalent.
\end{proof}



\begin{thebibliography}{99}

\bibitem{ADS03}
    D.\ V.\ Alekseevsky, A.\ J.\ Di Scala, Minimal homogeneous submanifolds of symmetric spaces, \emph{Lie groups and symmetric spaces}, Amer.\ Math.\ Soc.\ Transl.\ Ser.\ 2, vol.\ 210, Amer.\ Math.\ Soc., Providence, RI, 11--25 (2003).

\bibitem{ADS04}
    D.\ V.\ Alekseevsky, A.\ J.\ Di Scala, The normal holonomy group of
    K\"{a}hler submanifolds, \emph{Proc.\ London Math.\ Soc.} (3) \textbf{89}
    (2004), no. 1, 193--216.

\bibitem{BB01} J.~Berndt, M.~Br\"{u}ck, Cohomogeneity one actions on hyperbolic
    spaces, \emph{J.\ Reine Angew.\ Math.}\ \textbf{541} (2001), 209--235.

\bibitem{BCO03}
    J.\ Berndt, S.\ Console, C.\ Olmos, \emph{Submanifolds and holonomy},
    Chapman \& Hall/CRC Research Notes in Mathematics, \textbf{434}, Chapman \& Hall/CRC,
    Boca Raton, FL, 2003.

\bibitem{BD09} J.\ Berndt, J.\ C.\ D\'{\i}az-Ramos, Homogeneous
    hypersurfaces in complex hyperbolic spaces, \emph{Geom.\ Dedicata},
    \textbf{138} (2009), 129--150.

\bibitem{BD11a} J.~Berndt, J.~C.~D\'{\i}az-Ramos, Polar actions on the complex hyperbolic plane, \emph{Ann.\ Global Anal.\ Geom.} \textbf{43} (2013), 99--106.

\bibitem{BD11b} J.~Berndt, J.~C.~D\'{\i}az-Ramos, Homogeneous polar
    foliations of complex hyperbolic spaces, \emph{Comm.\ Anal.\
    Geom.} \textbf{20} (2012), 435--454.

\bibitem{BDT10} J.\ Berndt, J.\ C.\ D\'{\i}az-Ramos, H.\ Tamaru, Hyperpolar
    homogeneous foliations on symmetric spaces of noncompact type, \emph{J.\
    Differential Geom.}\ \textbf{86} (2010), 191--235.

\bibitem{BT07} J.~Berndt, H.~Tamaru, Cohomogeneity one actions on noncompact
    symmetric spaces of rank one, \emph{Trans.\ Amer.\ Math.\ Soc.}\
    \textbf{359} (2007), no. 7, 3425--3438.

\bibitem{BT} J.\ Berndt, H.\ Tamaru, Cohomogeneity one actions on
    symmetric spaces of noncompact type, \emph{J.\ Reine Angew.\ Math.}\ \textbf{683} (2013), 129--159.

\bibitem{BTV95} J.~Berndt, F.~Tricerri, L.~Vanhecke,
    \emph{Generalized Heisenberg  groups and Damek-Ricci harmonic
    spaces}, Lecture Notes in Mathematics \textbf{1598},
    Springer-Verlag, Berlin, 1995.

\bibitem{Bi06} L.\ Biliotti, Coisotropic and polar actions on
    compact irreducible Hermitian symmetric spaces, \emph{Trans.\ Amer.\ Math.\ Soc.}\
    \textbf{358} (2006), 3003--3022.

\bibitem{Da85} J.\ Dadok, Polar coordinates induced by actions of
    compact Lie groups, \emph{Trans.\ Amer.\ Math.\ Soc.} \textbf{288} (1985), 125--137.

\bibitem{D13} J.~C.~D\'{\i}az-Ramos, Polar actions on symmetric spaces, \emph{Recent Trends in Lorentzian Geometry}, Springer Proceedings in Mathematics \& Statistics, Springer, 315--334 (2013).

\bibitem{DD12} J.\ C.\ D\'{\i}az-Ramos, M.\ Dom\'{\i}nguez-V\'azquez, Inhomogeneous
    isoparametric hypersurfaces in complex hyperbolic spaces, \emph{Math.\ Z.}\
    \textbf{271} (2012), no.~3-4, 1037--1042.

\bibitem{DD13} J.~C.~D\'{\i}az-Ramos, M.~Dom\'{\i}nguez-V\'{a}zquez, Isoparametric hypersurfaces in Damek-Ricci spaces, \emph{Adv.\ Math.}\ \textbf{239} (2013), 1--17.

\bibitem{DK11} J.\ C.\ D\'{\i}az-Ramos, A.\ Kollross, Polar actions with a
    fixed point, \emph{Differential Geom.\ Appl.}\ \textbf{29} (2011), 20--25.

\bibitem{E96}
P.~B.~Eberlein, \emph{Geometry of nonpositively curved manifolds},
Chicago Lectures in Mathematics, University of Chicago Press, Chicago, IL, 1996.

\bibitem{G04} C.\ Gorodski, Polar actions on compact symmetric spaces
    which admit a totally geodesic principal orbit, \emph{Geom.\ Dedicata} \textbf{103}
    (2004), 193--204.

\bibitem{HPTT95} E.\ Heintze, R.\ S.\ Palais, C.-L.\ Terng,
    G.\ Thorbergsson, Hyperpolar actions on symmetric spaces,
    Geometry, topology, \& physics, Conf. Proc. Lecture Notes
    Geom. Topology, IV, Int. Press, Cambridge, MA, 214--245 (1995).

\bibitem{Ka53} F.\ I.\ Karpelevich, Surfaces of transitivity of semisimple
    group of motions of a symmetric space, \emph{Doklady Akad.\ Nauk SSSR}
    \textbf{93} (1953), 401--404.

\bibitem{k02} A.\ Kollross, A classification of hyperpolar and cohomogeneity
    one actions, \emph{Trans.\ Amer.\ Math.\ Soc.}\ \textbf{354} (2002),
    571--612.

\bibitem{K07} A.\ Kollross, Polar actions on symmetric spaces,
    \emph{J.\ Differential Geom.}\ \textbf{77} (2007), 425--482.

\bibitem{K09} A.\ Kollross, Low cohomogeneity and polar actions  on  exceptional compact Lie groups, \emph{Transform.\ Groups} \textbf{14} (2009),    387--415.

\bibitem{k11} A.\ Kollross, Duality of symmetric spaces and polar actions, \emph{J.\ of Lie Theory} \textbf{21} (2011), no.~4,  961--986.

\bibitem{KL} A.\ Kollross, A.\ Lytchak, Polar actions on
    symmetric spaces of higher rank, \emph{Bull.\ Lond.\ Math.\ Soc.}\ \textbf{45} (2013), no.~2, 341--350.

\bibitem{KrLy} L.\ Kramer, A.\ Lytchak, Homogeneous compact geometries,  	 arXiv:1205.2342 [math.GR].

\bibitem{lytchak} A.\ Lytchak, Polar foliations on symmetric spaces,
     	to appear in \emph{Geom.\ Funct.\ Anal.}

\bibitem{Mo55} G.~D.\ Mostow, Some new decomposition theorems for semi-simple
    groups, \emph{Mem.\ Amer.\ Math.\ Soc.} \textbf{14} (1955), 31--54.

\bibitem{mostow} G.~D.\ Mostow, On maximal subgroups of real Lie groups,
    \emph{Ann.\ of Math.}\ (2) \textbf{74} (1961), 503--517.

\bibitem{OV94} A.~L.~Onishchik, E.~B.~Vinberg (Eds.), \emph{Lie
    groups and Lie algebras III. Structure of Lie groups and Lie
    algebras}, Encyclopaedia of Mathematical Sciences, 41,
    Springer-Verlag, Berlin, 1994.

\bibitem{pt} R.\ Palais, C.-L.\ Terng, A general theory of canonical forms,
    \emph{Trans.\ Amer.\ Math.\ Soc.}\ \textbf {300} (1987), 771--789.

\bibitem{PT99} F.~Podest\`a, G.~Thorbergsson, Polar actions on
    rank-one symmetric spaces, \emph{J.\ Differential Geom.}
    \textbf{53} (1999), no.\ 1, 131--175.

\bibitem{T05} G.~Thorbergsson, Transformation groups and
    submanifold geometry,  \emph{Rend.\ Mat.\ Appl.}\ (7) 25 (2005), no. 1, 1--16.

\bibitem{T10} G.~Thorbergsson, Singular Riemannian
    foliations and isoparametric submanifolds, \emph{Milan J.\ Math.}\
    \textbf{78} (2010), no. 1, 355--370.

\bibitem{To07} D.~T\"oben, Singular Riemannian foliations on nonpositively
curved manifolds, \emph{Math.\ Z.} \textbf{255} (2007), 427--436.

\bibitem{wu} B.~Wu, Isoparametric submanifolds of hyperbolic
    spaces, \emph{Trans.\ Amer.\ Math.\ Soc.}\ \textbf{331} (1992), 609--626.

\end{thebibliography}
\end{document}